\pdfoutput=1
\documentclass[reqno]{amsart}
\usepackage{hyperref, bbm, amssymb, mathtools, url, enumerate}
\usepackage{amsmath}
\usepackage{amsthm}
\usepackage{tikz}
\usepackage{tikz-cd}
\usepackage[scr]{rsfso}
\usepackage{comment}
\usepackage{bm}
\usepackage{amssymb}
\usepackage{bbm}
\usetikzlibrary{matrix,arrows,decorations.pathmorphing, cd}
\usepackage[capitalise]{cleveref}
\usepackage{stmaryrd}
\usepackage[utf8]{inputenc}

\def\twocell[#1]{\arrow[#1, dash, phantom, "\Rightarrow"{scale=1.125, yshift=-.4pt, description, allow upside down, sloped, inner sep=0pt}]}

\pgfdeclarelayer{bg}
\pgfsetlayers{bg,main}
\usetikzlibrary{calc}

\newtheorem{theorem}{Theorem}[section]

\newtheorem{corollary}[theorem]{Corollary}
\newtheorem{lemma}[theorem]{Lemma}
\newtheorem{proposition}[theorem]{Proposition}
\newtheorem*{introthm}{Main Theorem}
\theoremstyle{definition}
\newtheorem{construction}[theorem]{Construction}

\newtheorem{definition}[theorem]{Definition}

\newtheorem{example}[theorem]{Example}

\newtheorem{remark}[theorem]{Remark}
\newtheorem*{claim*}{Claim}
\newtheorem{warn}[theorem]{Warning}

\newcommand{\qednow}{\pushQED{\qed}\qedhere\popQED}

\newcommand{\notehelper}[3]{\textcolor{#3}{$\blacksquare$}\marginpar{\ifodd\thepage\raggedright\else\raggedleft\fi\color{#3}\tiny \textbf{#2:} #1}}

\DeclareMathOperator{\Bb}{\mathcal{B}}
\DeclareMathOperator{\Cc}{\mathcal{C}}
\DeclareMathOperator{\Dd}{\mathcal{D}}
\DeclareMathOperator{\Ee}{\mathcal{E}}
\newcommand{\Ff}{\mathcal{F}}
\DeclareMathOperator{\Mm}{\mathcal{M}}
\DeclareMathOperator{\Nn}{\mathcal{N}}
\DeclareMathOperator{\Ii}{\mathcal{I}}

\newcommand{\EF}{E\Ff}

\renewcommand{\phi}{\varphi}
\renewcommand{\epsilon}{\varepsilon}

\renewcommand{\S}{{{\mathscr S}}}
\DeclareMathOperator{\Sp}{Sp}
\DeclareMathOperator{\Spc}{Spc}
\DeclareMathOperator{\Fin}{Fin}

\DeclareMathOperator{\Cat}{Cat}

\newcommand{\PrL}{\textup{Pr}^{\textup{L}}}

\DeclareMathOperator{\Hom}{Hom}
\DeclareMathOperator{\iHom}{\ul{\Hom}}
\DeclareMathOperator{\Fun}{Fun}
\DeclareMathOperator{\PSh}{PSh}

\DeclareMathOperator{\CMon}{CMon}

\DeclareMathOperator{\RelCat}{RelCat}

\DeclareMathOperator{\KK}{KK}
\DeclareMathOperator{\sep}{sep}
\DeclareMathOperator{\Ind}{Ind}

\newcommand{\catop}{^{\mathrm{op}}}
\newcommand{\op}{{\textup{op}}}

\renewcommand{\smallint}{{\textstyle\int}}
\newcommand{\essim}{{\textup{ess\,im}}}

\DeclareMathOperator{\unit}{\mathbbm{1}}

\DeclareMathOperator{\colim}{colim}

\DeclareMathOperator{\id}{id}
\DeclareMathOperator{\pr}{pr}

\DeclareMathOperator{\BC}{BC}

\DeclareMathOperator{\Orb}{Orb}

\DeclareMathOperator{\Glo}{Glo}

\DeclareMathOperator{\Nm}{Nm}

\DeclareMathOperator{\Nmadjdual}{\overline{\Nm}}

\newcommand{\Nfree}{N\textup{-free}}
\newcommand{\tors}[1]{#1\textup{-tors}}
\newcommand{\comp}[1]{#1\textup{-comp}}
\newcommand{\free}[1]{#1\textup{-free}}
\newcommand{\fr}{\textup{fr}}

\newcommand{\ulhelper}[2]{\underline{\setbox0=\hbox{$#1#2$}\dp0=1pt \box0\relax}}
\newcommand{\ul}[1]{{\mathpalette\ulhelper{#1}}\hbox{\rule[-2pt]{0pt}{0pt}}}

\newcommand{\finSets}{\mathbb{F}}

\newcommand{\ulfinPsets}{\ul{\finSets}_T^{\Orbital}}

\newcommand{\Orbital}{P}

\newcommand{\blank}{{\textup{--}}}

\newcommand{\tcatUn}[1]{\mathop{\hfuzz=10pt\hbox to 0pt{$\textstyle\bm\int$}\kern.3pt\raise.2pt\hbox to 0pt{$\textstyle\bm\int$}\lower.2pt\hbox to 0pt{$\textstyle\bm\int$}\kern.3pt\hbox to 0pt{$\textstyle\bm\int$}\kern-.1pt\raise.1pt\hbox{\color{white}$\textstyle\int$}}}
\newcommand{\GammaS}{{\Gamma\kern-1.5pt\mathscr S}}
\newcommand{\mySp}{{\mathscr S\kern-2ptp}}
\newcommand{\mathscrGr}{{\mathscr G\kern-1.25ptr}}

\DeclareMathOperator{\gl}{gl}

\DeclareMathOperator{\Span}{Span}

\let\smashp=\wedge

\newcommand{\iso}{\xrightarrow{\;\smash{\raisebox{-0.5ex}{\ensuremath{\scriptstyle\sim}}}\;}}

\newcommand{\pushoutsign}{\hspace{0.2ex}\tikz[baseline=(po.base)]{\draw (0,0) ++ (2.45ex,-1.45ex) -- ++(0ex,1ex) -- ++ (1ex,0ex);\node (po) at (0,0) {\phantom{x}};}}
\tikzcdset{
        pullback/.style = {phantom, "\lrcorner"{very near start}},
	pushout/.style = {dash, phantom, "\vphantom{x}\smash{\pushoutsign}", start anchor=center, end anchor=center}
}

\newcommand\noloc{
	\nobreak
	\mspace{6mu plus 1mu}
	{:}
	\nonscript\mkern-\thinmuskip
	\mathpunct{}
	\mspace{2mu}
}

\makeatletter
\def\nop@r@m{nop@r@m}
\newif\ifbr@cketed
\def\br@cketed{}
\let\@uldcite=\cite
\def\st@rcite#1#2{\@uldcite[#2]{#1}}
\def\@mscite#1{\@ifstar{\st@rcite{#1}}{\ifbr@cketed\@uldcite[\br@cketed]{#1}\else\@uldcite{#1}\fi}}
\renewcommand{\cite}[1][\nop@r@m]{\def\temp{#1}\ifx#1\nop@r@m\relax\br@cketedfalse\else\br@cketedtrue\fi\gdef\br@cketed{#1}\@mscite}
\makeatother
\let\cites=\cite

\title{The Adams isomorphism revisited}
\author{Bastiaan Cnossen}
\address{B.C.: Fakultät für Mathematik, Universität Regensburg, 93040 Regensburg, Germany}
\author{Tobias Lenz}
\address{T.L.:{\hskip0pt minus 1pt} Mathematical{\hskip0pt minus 1pt} Institute,{\hskip0pt minus 1pt} University{\hskip0pt minus 1pt} of{\hskip0pt minus 1pt} Utrecht,{\hskip0pt minus 1pt} Budapestlaan{\hskip0pt minus 1pt} 6,{\hskip0pt minus 1pt} 3584{\hskip0pt minus 1pt} CD{\hskip0pt minus 1pt} Utrecht, The Netherlands}
\author{Sil Linskens}
\address{S.L.:{\hskip0pt minus 1pt} Mathematisches{\hskip0pt minus 1pt} Institut,{\hskip0pt minus 1pt} Rheinische{\hskip0pt minus 1pt} Friedrich-Wilhelms-Universit\"at{\hskip0pt minus 1pt} Bonn,{\hskip0pt minus 1pt} Endenicher{\hskip0pt minus 1pt} Allee{\hskip0pt minus 1pt} 60,{\hskip0pt minus 1pt} 53115{\hskip0pt minus 1pt} Bonn,{\hskip0pt minus 1pt} Germany}
\begin{document}
	\begin{abstract}
		We establish abstract Adams isomorphisms in an arbitrary equivariantly presentable equivariantly semiadditive global category. This encompasses the well-known Adams isomorphism in equivariant stable homotopy theory, and applies more generally in the settings of $G$-Mackey functors, $G$-global homotopy theory, and equivariant Kasparov categories.
	\end{abstract}
	\maketitle

	\begingroup\parskip=0.5\parskip
	\setcounter{tocdepth}{1}
	\tableofcontents
	\endgroup

	\section{Introduction}
	In \cite{Adams1984}, Adams proved a surprising result in equivariant stable homotopy theory: given a finite group $G$ and a finite pointed $G$-CW-complex on which $G$ acts freely away from the basepoint, there is an equivalence between the $G$-orbit spectrum and the (derived) $G$-fixed spectrum of the suspension spectrum $\Sigma^{\infty} X$. This sits in stark contrast to the unstable or na\"ively stable situation, where the $G$-fixed points of $X$ are just trivial unless $G=e$. In \cite{LMS}, Lewis, May, and Steinberger extended this equivalence to all so-called \emph{$G$-free} genuine $G$-spectra, and called it the \textit{Adams isomorphism}. Since then, it has become an indispensable tool in equivariant homotopy theory, for example allowing one to calculate the $G$-equivariant stable homotopy groups of a free $G$-CW complex $X$ simply as the non-equivariant stable homotopy groups of its $G$-orbits $X/G$.

	The goal of the present article is to show that the Adams isomorphism is a purely formal consequence of \textit{equivariant semiadditivity}, an analogue of semiadditivity in the context of parametrized higher category theory \cite{exposeI} introduced by the authors \cite{CLL_Global} building on work of Nardin \cite{nardin2016exposeIV}. Our treatment can be viewed as a highly abstract version of the classical proof of the Adams isomorphism, which is a clever reduction to an application of the \textit{Wirthm\"uller isomorphism} for equivariant spectra. The notion of equivariant semiadditivity is designed to provide an abstract formulation of the Wirthm\"uller isomorphism, and the key observation of this paper is that the Adams isomorphism and its proof admit similar abstract formulations. This in particular provides analogues of the Adams isomorphism in various new settings, like that of $G$-global homotopy theory, $G$-Mackey functors, and equivariant Kasparov categories.

    The categorical framework we will work in is that of \textit{global categories}:\footnote{Throughout this article, we will say `category' for `$\infty$-category.'} families of categories $\Cc\colon \Glo\catop \to \Cat$ indexed by the $(2,1)$-category $\Glo$ of finite groups, group homomorphisms, and conjugations. We will require $\Cc$ to be \textit{equivariantly presentable} in the sense of \cite{CLL_Clefts}, meaning that $\Cc$ factors through the subcategory $\PrL \hookrightarrow \Cat$ of presentable categories and that for all injective group homomorphisms $i\colon H \hookrightarrow G$ the restriction functors $i^*\colon \Cc(G) \to \Cc(H)$ admit left adjoints $i_!\colon \Cc(H) \to \Cc(G)$ satisfying a base change condition. We will further require $\Cc$ to be \textit{equivariantly semiadditive}: each category $\Cc(G)$ is semiadditive in the usual sense and for every injection $i\colon H \hookrightarrow G$ a certain \textit{norm map} $\Nm_i\colon i_! \to i_*$ to the \emph{right} adjoint $i_*$ of $i^*$ is invertible.

    To state the abstract Adams isomorphism in this context, consider a normal subgroup $N$ of a finite group $G$. As a consequence of equivariant presentability of $\Cc$, the category $\Cc(G)$ is canonically tensored over $G$-spaces, and in analogy with the usual definition for genuine $G$-spectra we say that an object $X \in \Cc(G)$ is \textit{$N$-free} if the canonical map $X\otimes E\mathcal F_N \to X\otimes1 = X$ is an equivalence. Here $\EF_N$ is the universal $G$-space for the family $\Ff_N$ of subgroups $H \leqslant G$ such that $H \cap N = e$ (see \Cref{def:NBorelObjects}). On the $N$-free $G$-objects, one can define an \textit{$N$-orbit functor}
	\[
	-/N\colon \Cc(G)_{\Nfree} \to \Cc(G/N),
	\]
	which is a partial left adjoint of the inflation functor $\Cc(G/N) \to \Cc(G)$ associated to the quotient morphism $G \to G/N$. The inflation functor moreover admits a right adjoint $(-)^N\colon \Cc(G) \to \Cc(G/N)$, called the \textit{$N$-fixed point functor}, and our main result is the following comparison between these adjoints:

    \begin{introthm}[Abstract Adams isomorphism for global categories, \Cref{thm:GlobalFormalAdamsIso}]
        \label{introthm:FormalAdamsIso}
    Let $\Cc$ be an equivariantly presentable and equivariantly semiadditive global category. Let $N \trianglelefteqslant G$ be a normal subgroup of a finite group $G$, and assume that the $N$-fixed point functor $(-)^N\colon \Cc(G) \to \Cc(G/N)$ preserves colimits. Then there is a natural equivalence
        \[
        \Nm \colon X/N \iso X^N
        \]
    	in $\Cc(G/N)$ for every $N$-free object $X \in \Cc(G)_{\Nfree}$.
    \end{introthm}

    Specializing the theorem, we obtain abstract Adams isomorphisms for various global categories of interest:

	\begin{enumerate}[(1)]
		\item For $\Cc(G) = \mySp_G$, the category of genuine $G$-spectra (\Cref{ex:EquivHtpyTheory}), this precisely recovers the classical Adams isomorphism \cite{LMS}*{Theorem~II.7.1}.
		\item For $\Cc(G) = \mySp_G^{\gl}$, the category of $G$-global spectra (\Cref{ex:GlobalEquivHtpyTheory}), this provides a \emph{global Adams isomorphism} refining the classical Adams isomorphism and generalizing the result for $G=N$ appearing in \cite{tigilauri}.
		\item There are also `non-group completed' versions of the previous two examples, yielding Adams isomorphism for special $G$-equivariant or $G$-global $\Gamma$-spaces, see Example~\ref{ex:gamma-Adams}.
		\item If $\Ee$ is any semiadditive category, the global category of \emph{Mackey functors} in $\Ee$, given by $\Cc(G) = \mathrm{Mack}_{\Ee}(G) \coloneqq \Fun^{\times}(\Span(\Fin_G),\Ee)$, satisfies the assumptions of the theorem, see Example~\ref{ex:MackeyFunctors}. In particular, taking $\Ee$ to be the $1$-category of abelian groups, this yields Adams isomorphisms for classical Mackey functors, while for $\Ee$ the derived category of abelian groups we obtain an Adams isomorphism for Kaledin's \emph{derived Mackey functors} \cite{DerMack11}.
		\item The theorem further applies to $\Cc(G) = \KK^G$, the $G$-Kasparov K-theory category defined by Bunke, Engel, and Land \cite{BEL2023Kasparov}, which we recall in \Cref{ex:Kasparov}.
	\end{enumerate}

	In fact, we prove a more general statement in which the category $\Glo$ is replaced by a small category $T$ and the role played by the wide subcategory $\Orb \subseteq \Glo$ of \textit{injective} homomorphisms in the definitions of equivariant presentability and semiadditivity is taken by a so-called \textit{atomic orbital} subcategory $P \subseteq T$, see \Cref{thm:adams-gen}.

    \subsection*{Related work} Given that the Adams isomorphism is both a fundamental and yet somewhat surprising feature of equivariant stable homotopy theory, it is natural to ask for the key properties necessary to establish Adams isomorphisms in other contexts. The quest for such \emph{formal Adams isomorphisms} goes back at least to \cite{wirthrev}, where May pointed out the formal similarity of the Adams isomorphism to the Wirthm\"uller isomorphism, but noted that unlike the latter it so far resisted an axiomatic treatment.

    A first step in this direction was taken by Hu \cite{Hu2003Duality}, who provided a very general duality result in parametrized homotopy theory encompassing both the Wirthmüller isomorphism as well as the Adams isomorphism in equivariant stable homotopy theory for compact Lie groups. While her treatment provided part of the original inspiration for our approach, the nature of her argument is ultimately geometric rather than categorical, employing a parametrized version of the Pontryagin-Thom construction. An analogue of her argument in the context of equivariant motivic homotopy theory was indicated by Hoyois \cite{Hoyois2017Motivic}*{Section~1.4.1}.

    An entirely different approach to formal Adams isomorphisms was given by Sanders \cite{Sanders2019Adams}, based on the notion of the \textit{compactness locus} of a geometric functor $f^*$ between rigidly compactly generated tensor triangulated categories. This setup is somewhat orthogonal to ours: while Sanders' approach works for a single functor $f^*$ without need for a parametrized category in the background, we on the other hand do not require stability or any sort of monoidal structure. In fact, several of our key examples do not fit into Sanders' framework despite having natural symmetric monoidal structures: the tt-category of global spectra is not rigid (Remark~\ref{rk:global-not-rigid}), while examples like $G$-equivariant $\Gamma$-spaces or $G$-Mackey functors are not even stable. Nevertheless, we will show in \Cref{sec:Comparison_Sanders} that our abstract Adams isomorphism agrees with that of Sanders whenever the two frameworks overlap.

	In addition to these general approaches, also the classical Adams isomorphism for genuine $G$-spectra has been revisited several times. In particular, Reich and Varisco \cite{ReichVarisco2016Adams} first lifted the construction of the classical Adams isomorphism from the homotopy category to the $\infty$-categorical level, by giving an explicit and natural model of the map in the $1$-category of orthogonal spectra.

	A treatment in the language of parametrized category theory was first given by Quigley and Shah \cite[Remark~4.29]{QuigleyShay2021Tate}, who introduced a \textit{parametrized Tate construction} for genuine equivariant spectra, with the classical Adams isomorphism as a special case. Similarly to our approach below, their construction is based on the norm maps of Hopkins and Lurie \cite{hopkins2013ambidexterity}.

	A treatment of the Adams isomorphism in equivariant motivic homotopy theory was given by Gepner and Heller \cite{Gepner_Heller}.

	\subsection*{Outline} In Section~\ref{sec:recollections} we recall relevant background material on parametrized higher category theory. Section~\ref{sec:families} introduces and studies \emph{torsion objects} in param\-etrized categories, generalizing the notion of $N$-free $G$-spectra. Using this language, we then state and prove our general abstract Adams isomorphism in Section~\ref{sec:adams-iso-gen}. In Section~\ref{sec:Adams_Iso_Equivariant} we specialize this result to the Adams isomorphism for global categories stated above, and we discuss various examples.

	The paper ends with three short appendices. In Appendix~\ref{sec:cocont} we prove a general statement about cocontinuity of the right adjoints to restrictions for several universal examples of parametrized higher categories, which is in particular used in establishing the Adams isomorphisms for global and equivariant $\Gamma$-spaces. In Appendix~\ref{sec:Comparison_Sanders}, we compare our abstract Adams isomorphism to that of Sanders \cite{Sanders2019Adams}. Finally, Appendix~\ref{sec:BC} collects some aspects of the calculus of mates used throughout the paper.

	\subsection*{Acknowledgements}
     B.C.\ thanks Ulrich Bunke and Benjamin Dünzinger for discussions on equivariant Kasparov categories. The authors would like to thank the referee for helpful comments, which led to the inclusion of Appendix~\ref{sec:BC}.

	During the preparation of this article, B.C.\ was supported by the Max Planck Institute for Mathematics in Bonn and the SFB 1085 `Higher Invariants' in Regensburg, funded by the DFG.

	S.L.\ is an associate member of the Hausdorff Center for Mathematics at the University of Bonn, and is supported by the DFG Schwerpunktprogramm 1786 `Homotopy Theory and Algebraic Geometry' (project ID SCHW 860/1-1).

	\section{Recollections on parametrized higher categories}\label{sec:recollections}
	In this article, we will freely use the language of parametrized higher category theory, as established by \cites{exposeI, shah2021parametrized, nardin2016exposeIV} and, from the perspective of categories internal to $\infty$-topoi, by \cites{martini2021yoneda, martiniwolf2021limits, martiniwolf2022presentable}. We will follow our conventions from \cite{CLL_Global}. Throughout this section, we fix a small category $T$.

	\begin{definition}
		A \textit{$T$-category} is a functor $\Cc\colon T\catop \to \Cat$. If $\Cc$ and $\Dd$ are $T$-categories, then a \textit{$T$-functor} $F\colon \Cc \to \Dd$ is a natural transformation from $\Cc$ to $\Dd$. The category $\Cat_T$ of $T$-categories is defined as the functor category $\Cat_T\coloneqq\Fun(T\catop, \Cat)$.
	\end{definition}

	\begin{remark}
		As a category of functors into a Cartesian closed category, $\Cat_T$ is again Cartesian closed. We write $\ul\Fun_T$ for the internal hom objects.
	\end{remark}

	\begin{example}
		We write $\Glo$ for the $(2,1)$-category of finite groups, homomorphisms, and with conjugations as $2$-cells; equivalently, this is the $(2,1)$-category of finite connected $1$-groupoids.

		We will refer to $\Glo$-categories as \emph{global categories} throughout.
	\end{example}

	\begin{example}\label{ex:Orb}
		We write $\Orb\subseteq\Glo$ for the wide subcategory spanned by the \emph{injective} group homomorphisms. Restricting along the inclusion, every global category has an underlying $\Orb$-category.
	\end{example}

	Any $T$-category can be equivalently viewed as a limit preserving functor $\PSh(T)^\op\to\Cat$; in particular, we will frequently evaluate $T$-categories at general presheaves over $T$.

	\begin{example}[cf.~\cite{CLL_Global}*{Example~2.1.11 and Remark~2.1.16}]
		We have a limit preserving functor $\ul\Spc_T=\PSh(T)_{/\bullet}\colon\PSh(T)^\op\to\Cat$ sending a presheaf $X$ to $\PSh(T)_{/X}$, with contravariant functoriality via pullback. As a functor $T^\op\to\Cat$ this can be equivalently described as $\ul\Spc_T(A)\simeq\PSh(T_{/A})$, with functoriality via restriction along the postcomposition maps $T_{/A} \to T_{/B}$ for $A \to B$.
	\end{example}

	\begin{example}\label{ex:Orb-spaces-G-SSet}
		For every finite group $G$, the slice category $\Orb_{/G}$ is equivalent to the $1$-category $\Orb_G$ of transitive $G$-sets (via sending an injective homomorphism $\phi\colon H\to G$ to $G/\text{im}(\phi)$). By Elmendorf's theorem \cite{elmendorf}, we can therefore identify each individual $\ul\Spc_{\Orb}(G)$ with the $\infty$-category of $G$-spaces. In fact, as explained in \cite{CLL_Clefts} we can bundle up these equivalences to give an explicit model $\ul\S$ of $\ul\Spc_{\Orb}$ as follows:

		We have a functor $\Fun(B(-),\text{SSet})\colon\Glo^\op\to\Cat$ sending a group $G$ to the category $\Fun(BG,\text{SSet})$ of $G$-objects in the $1$-category of simplicial sets, with the obvious functoriality. For every such $G$, we can equip $\Fun(BG,\text{SSet})$ with the \emph{$G$-equivariant weak equivalences}, i.e.~those maps whose geometric realization is a $G$-equivariant homotopy equivalence. The restrictions are then clearly homotopical, so this defines a functor from $\Glo^\op$ into the category of relative categories. Postcomposing with $\infty$-categorical localization, we therefore obtain a global category $\ul{\mathscr S}$, and \cite{CLL_Clefts}*{Theorem~5.5} provides the desired equivalence $\ul\S\simeq\ul\Spc_{\Orb}$.
	\end{example}

	\subsection{Orbitality and parametrized (co)products}

	\begin{definition}[{\cite{CLL_Global}*{Definition~4.2.2, Definition~4.2.11}}]
		Let $P \subseteq T$ be a wide subcategory. We say that $P$ is \textit{orbital} if for every pullback diagram
		\begin{equation}
			\label{eq:PullbackSquareOrbitality}
			\begin{tikzcd}
				A' \dar[swap]{p'} \rar{\alpha} \drar[pullback] & A \dar{p} \\
				B' \rar{\beta} & B
			\end{tikzcd}
		\end{equation}
		in $\PSh(T)$, with $A,B,B' \in T$ and $p\colon A \to B$ in $P$, the morphism $p'\colon A' \to B'$ can be decomposed as a disjoint union $(p_i)_{i=1}^n\colon \bigsqcup_{i=1}^n A_i \to B'$ for morphisms $p_i\colon A_i \to B'$ in $P$.

		If $P$ is orbital, we let $\ulfinPsets \subseteq \ul{\Spc}_T=\PSh(T)_{/\bullet}$ denote the full parametrized subcategory spanned at $B \in T$ by those morphisms $A \to B$ in $\PSh(T)$ which are equivalent to disjoint unions $(p_i)_{i=1}^n\colon \bigsqcup_{i=1}^n A_i \to B$ for morphisms $p_i\colon A_i \to B$ in $P$. Orbitality of $P$ guarantees that this is well-defined.

		We say a map $f\colon A\rightarrow B$ in $\PSh(T)$ is in $\ulfinPsets$ if it is an object of $\ulfinPsets(B)$.
	\end{definition}

	\begin{example}
		The $(2,1)$-category $\Glo$ is orbital (as a subcategory of itself): the finite coproduct completion $\mathbb F_{\Glo}\coloneqq\mathbb F_{\Glo}^{\Glo}\subseteq\PSh(\Glo)$ is simply the $(2,1)$-category of finite groupoids, so it has all pullbacks, and as it moreover contains all representables, limits in $\mathbb F_{\Glo}$ are already limits in $\PSh(\Glo)$.
	\end{example}

	\begin{example}
		The wide subcategory $\Orb\subseteq\Glo$ from Example~\ref{ex:Orb} is an orbital subcategory of $\Glo$, see \cite{CLL_Global}*{Example 4.2.5}.
	\end{example}

	\begin{example}
		A $T$-category $\Cc$ is said to \textit{admit finite $P$-coproducts} if it admits $\ulfinPsets$-indexed colimits in the sense of \cite{CLL_Global}*{Definition 2.3.8}. Explicitly, this means that the restriction functor $p^*\colon \Cc(B) \to \Cc(A)$ admits a left adjoint $p_!\colon \Cc(A) \to \Cc(B)$ for every morphism $p\colon A \to B$ in $\ulfinPsets$, and for every pullback square \eqref{eq:PullbackSquareOrbitality} the \emph{Beck-Chevalley transformation}\footnote{We refer the reader to Appendix~\ref{sec:BC} for background on Beck-Chevalley transformations and some aspects of their rich calculus.}
		\[
			p'_!\alpha^* \xrightarrow{\eta} p'_!\alpha^*p^*p_!\simeq p'_!p^{\prime*}\beta^*p_!\xrightarrow{\epsilon} \beta^*p_!
		\] is an equivalence of functors $\Cc(A) \to \Cc(B')$, see \cite[Lemma~2.3.13]{CLL_Global}. Dually, we get a notion of finite $P$-products.

		We say that a $T$-functor $F\colon\Cc\to\Dd$ \emph{preserves finite $P$-coproducts} if for every $p\colon A\to B$ in $\ulfinPsets$ the Beck-Chevalley map $p_!F_A\to F_Bp_!$ (the mate of the naturality constraint $F_Af^*\simeq f^*F_B$) is an equivalence.
	\end{example}

	By \cite{CLL_Global}*{Proposition~4.2.14}, if $\Cc$ has finite $P$-coproducts, then each $\Cc(B)$ for $B\in\PSh(T)$ in particular has finite coproducts in the usual sense and each $f^*\colon\Cc(B)\to\Cc(A)$ preserves them. Similarly, if $F\colon\Cc\to\Dd$ preserves finite $P$-coproducts, then each $\Cc(A)\to\Dd(A)$ preserves finite coproducts in the usual sense, see Proposition~4.2.15 of \emph{op.~cit.}

	\subsection{Semiadditivity and stability} In order to talk about parametrized notions of semiadditivity, we need an extra condition on our orbital subcategory $P$:

	\begin{definition}[{\cite[Definition~4.3.1]{CLL_Global}}]
		An orbital subcategory $P \subseteq T$ is said to be \textit{atomic orbital} if for every morphism $p\colon A \to B$ in $P$ the diagonal $\Delta: A \to A \times_B A$ in $\PSh(T)$ is a disjoint summand inclusion, i.e. equivalent to a map of the form $A \hookrightarrow A \sqcup C$ for some presheaf $C \in \PSh(T)$.
	\end{definition}

	\begin{example}
		The subcategory $\Orb\subseteq\Glo$ is in fact even atomic orbital, see \cite{CLL_Global}*{Example~4.3.3}.
	\end{example}

	For the rest of this section, let $P$ be an atomic orbital subcategory of $T$. Let $\Cc$ be a pointed $T$-category, and assume that $\Cc$ admits finite $P$-products and finite $P$-coproducts. Then one can define for every morphism $p\colon A\to B$ in $\ul{\mathbb F}^P_T$ a \textit{norm map}
	\[
	\Nm_p\colon p_! \to p_*
	\]
	of functors $\Cc(A) \to \Cc(B)$, see \cite[Construction~4.3.6]{CLL_Global}.

	\begin{definition}
		A pointed $T$-category $\Cc$ admitting both finite $P$-coproducts and finite $P$-products is said to be \textit{$P$-semiadditive} if the norm map $\Nm_p$ is an equivalence for every $p\colon A\rightarrow B$ in $\ul{\mathbb F}^P_T$ or equivalently \cite{CLL_Global}*{Corollary~4.5.3} for those $p$ where in addition $B$ is a finite disjoint union of representables.

		If $\Cc$ is in addition \emph{fiberwise stable}, i.e.~it factors through the category of stable categories and exact functors, then we will call it \emph{$P$-stable}.
	\end{definition}

	A functor of $P$-semiadditive categories preserves finite $P$-coproducts if and only if it preserves finite $P$-products, see  \cite{CLL_Global}*{Proposition~4.6.14}; in this case, we call it \emph{$P$-semiadditive}.

	\begin{example}
		Specializing this to $\Orb\subseteq\Glo$ yields notions of {$\Orb$-semiadditivity} and $\Orb$-stability for global categories. As in \cite{CLL_Global,CLL_Clefts}, we will refer to these concepts as \emph{equivariant semiadditivity} and \emph{equivariant stability}.
	\end{example}

	\subsection{Presentability} In \cite{CLL_Clefts}*{Definition~4.3}, we introduced various degrees of parametrized presentability for $T$-categories, encoded in the choice of a so-called \emph{cleft} $S\subseteq T$. For the present article we will only be interested in the special cases where either $S=P\subseteq T$ is an atomic orbital subcategory \cite{CLL_Clefts}*{Example~3.7} or $S=T$ \cite{CLL_Clefts}*{Example~3.5}:

	\begin{definition}
		We say that a $T$-category $\Cc$ is \textit{$P$-presentable} if it admits finite $P$-coproducts and is \emph{fiberwise presentable} in the sense that the functor $\Cc\colon T^{\op}\rightarrow \Cat$ factors through the subcategory $\PrL\subseteq \Cat$.
	\end{definition}

	\begin{definition}
		We say that a $T$-category $\Cc$ is \textit{$T$-presentable} if it is fiberwise presentable and \emph{$T$-cocomplete}, i.e.~for every $f\colon A\to B$ in $T$ (or equivalently in $\PSh(T)$) the restriction $f^*$ admits a left adjoint $f_!$ satisfying base change.
	\end{definition}

	Note that if $P=T$ itself is atomic orbital, this indeed agrees with the previous definition by \cite{CLL_Global}*{Proposition 4.2.14}.

	\begin{example}
		The $T$-category $\ul\Spc_T$ is $T$-presentable. More generally, if $I$ is any small $T$-category, then the internal hom $\ul\Fun_T(I,\ul\Spc_T)$ is $T$-presentable. (In fact, \cite{martiniwolf2022presentable}*{Definition~6.2.1} \emph{defined} $T$-presentable categories as suitable localizations of such functor categories, with the equivalence to the previous definition being established as Theorem~6.2.4 of \emph{op.\ cit.})
	\end{example}

	\begin{remark}\label{rk:basechange}
		Already fiberwise presentability ensures that the restriction functors $f^*$ admit right adjoints $f_*$. If $\Cc$ is $T$-presentable, then these right adjoints again satisfy a Beck-Chevalley condition, i.e.~$\Cc$ is also \emph{$T$-complete}. We caution the reader, however, that the Beck-Chevalley condition for the right adjoints does \emph{not} hold for general $P$-presentable $T$-categories.
	\end{remark}

	Specializing the above, we get notions of \emph{$\Orb$-presentable} and \emph{$\Glo$-presentable} global categories; below, we will refer to these as \emph{equivariantly presentable} and \emph{globally presentable}, respectively.

	\begin{example}\label{ex:equiv-spectra}
		For any finite group $G$, we define $\mySp_G$, the ($\infty$-)category of \emph{genuine $G$-spectra}, as the localization of the $1$-category $\Fun(BG,\text{Sp}^\Sigma)$ of symmetric spectra with $G$-action at the \emph{$G$-stable weak equivalences} of \cite{hausmann-equivariant}*{Definition~3.25}. These fit together into a global category $\ul\mySp$, with structure maps given as left derived functors of the evident restriction functors between the $1$-categorical models.

		In more detail, $G\mapsto\Fun(BG,\text{Sp}^\Sigma)$ becomes a global $1$-category via precomposition. For every $G$, we can consider the full subcategory $\Fun(BG,\text{Sp}^\Sigma)^\text{cof}$ of those $G$-symmetric spectra that are cofibrant in the equivariant projective model structure of \cite{hausmann-equivariant}*{Theorem~4.8}; the restriction maps then preserve these subcategories and are homotopical in the above weak equivalences when restricted to them by \cite{CLL_Clefts}*{Lemma~9.3}, so that we obtain a functor $\Fun(-,\text{Sp}^\Sigma)^\text{cof}$ from $\Glo$ into the $2$-category of relative $1$-categories. Postcomposition with the localization functor then yields the global category $\ul{\mySp}$.

		We call this the \emph{global category of equivariant spectra}; it is equivariantly presentable and equivariantly stable: in fact, \cite{CLL_Clefts}*{Theorem~9.4} makes precise that it is the \emph{free} equivariantly presentable and stable global category on one generator.
	\end{example}

	\begin{example}\label{ex:global-spectra}
		We can also localize $\Fun(BG,\text{Sp}^\Sigma)$ at the \emph{$G$-global stable weak equivalences} of \cite{g-global}*{Definition~3.1.28}, i.e.~those maps $f$ such that the \emph{underived} restriction $\phi^*f$ is an $H$-stable weak equvialence for every $\phi\colon H\to G$, yielding the category $\mySp^\text{gl}_G$ of \emph{$G$-global spectra}. These again fit together into a global category $\ul\mySp^\text{gl}$, which we call the \emph{global category of global spectra}. It is the free \emph{globally} presentable equivariantly stable global category on one generator, see~\cite{CLL_Global}*{Theorem~7.3.2}. Its underlying category $\mySp^\text{gl}_e$ is the usual category of global spectra (with respect to finite groups) in the sense of \cite{schwede2018global}, see~\cite{hausmann-global}*{Theorem~5.3}.
	\end{example}

	\begin{example}\label{ex:gamma}
		For any finite group $G$, we can consider the ($\infty$-)category $\GammaS_*^\text{spc}(G)$ of \emph{special $\Gamma$-$G$-spaces} in the sense of Shimakawa \cite{shimakawa,shimakawa-simplify}, defined as a certain localization of the category of reduced functors from the category of finite pointed sets into the $1$-category of $G$-simplicial sets. For varying $G$, these again fit together into a global category $\ul\GammaS_*^\text{spc}$, see \cite{CLL_Clefts}*{paragraph after Lemma~7.14} for details. By Theorem~7.17 of \emph{op.~cit.}, this is the free equivariantly presentable and equivariantly semiadditive global category.

		Similarly, \cite{CLL_Global}*{Theorem~5.3.5} describes the free globally presentable equivariantly semiadditive global category $\ul\GammaS_*^\text{gl, spc}$ in terms of the \emph{$G$-global special $\Gamma$-spaces} of \cite{g-global}*{Section~2.2}.
	\end{example}

    \section{Torsion and complete objects}\label{sec:families}

    In the equivariant stable homotopy theory of a finite group $G$, there are two frequently used categories of equivariant spectra: the category $\mySp^{BG}$ of Borel $G$-spectra and the category $\mySp_G$ of genuine $G$-spectra. The forgetful functor $\mySp_G \to \mySp^{BG}$ admits fully faithful left and right adjoints $\mySp^{BG} \hookrightarrow \mySp_G$. The essential image of the left adjoint consists of the \textit{Borel-free $G$-spectra}: those $X$ such that the canonical map
    \[
    X \otimes EG \to X
    \]
    is an equivalence; here $EG$ is a free $G$-space whose underlying space is contractible and $- \otimes EG$ denotes the tensoring of a $G$-spectrum with the $G$-space $EG$. Equivalently, a $G$-spectrum $X$ is Borel-free if and only the $H$-geometric fixed points $\Phi^H(X)$ are trivial for every nontrivial subgroup $H \leqslant G$, which explains the use of the word `free'. Dually, the essential image of the right adjoint consists of the \textit{Borel-complete $G$-spectra}: those $X$ such that the canonical map $X\rightarrow X^{EG}$ is an equivalence, or equivalently those $X$ such that genuine and homotopy $H$-fixed points agree for every subgroup $H \leqslant G$.

    These two notions of Borel $G$-spectra have generalizations to an arbitrary family $\Ff$ of subgroups of $G$, that is, a collection of subgroups closed under conjugacy and passing to smaller subgroups. Let $\EF$ denote the \emph{universal $G$-space} for $\Ff$,\footnote{This is sometimes also referred to as the \emph{classifying $G$-space} for $\Ff$, see e.g.~\cite{lueck}; our terminology here follows \cite{LMS}*{Definition II.2.10}.} characterized by
    \[
        (\EF)^H = \begin{cases} 1 & H \in \Ff \\ \emptyset & H \notin \Ff.
        \end{cases}
    \]
    Following \cite{MNN2017Nilpotence}, we say that a $G$-spectrum $X$ is \textit{$\Ff$-torsion} if the map $X \otimes \EF \to X$ is an equivalence; other names used in the literature are \emph{$\Ff$-spectra} \cite{LMS}*{Definition~II.2.1}, \emph{$\mathcal F$-objects} \cite{mandell-may}*{II.6.1}, and \emph{$\Ff$-free spectra} \cite{greenlees-may}*{IV.17}. Dually, $X$ is called \textit{$\Ff$-complete} if $X \iso X^{\EF}$. The Borel-free and Borel-complete $G$-spectra are recovered by letting $\Ff$ consist of only the trivial subgroup. More generally, if $N\trianglelefteqslant G$ is normal, taking $\Ff$ to be those subgroups $H \leqslant G$ such that $H \cap N = e$ gives the notion of \textit{$N$-free $G$-spectra} discussed in the introduction.

    The goal of this section is to provide a general definition of torsion and complete objects in a parametrized category. Throughout this section, we let $P$ be an arbitrary small ($\infty$-)category.

	\begin{definition}
		For an object $A \in P$, a \textit{$P$-family over $A$} is a (non-empty) sieve of $P_{/A}$, i.e.\ a (non-empty) full subcategory $\Ff \subseteq P_{/A}$ satisfying the property that for every morphism in $P_{/A}$
		\[
		\begin{tikzcd}
			B \ar{rr} \drar[swap]{p} && C \dlar{q} \\
			& A,
		\end{tikzcd}
		\]
		if $q$ is in $\Ff$ then so is $p$.
	\end{definition}

	\begin{example}
		\label{ex:Trivial}
		The whole subcategory $\mathcal A\ell\ell_A \coloneqq P_{/A}$ is always a $P$-family over $A$.
	\end{example}

	\begin{example}
		\label{ex:FamilyOfSubgroups}
		Let $P = \Orb$ be the global orbit category, see \Cref{ex:Orb}. Then for a finite group $G$, $\Orb$-families over $G$ correspond under the equivalence $\Orb_{/G}\simeq\Orb_G$ (Example~\ref{ex:Orb-spaces-G-SSet}) to \emph{families of subgroups} of $G$, i.e.~non-empty collections of subgroups of $G$ closed under taking subconjugates.
	\end{example}

	\begin{example}
		\label{ex:ProperSubobjects}
		Assume that $P$ admits no non-trivial retracts; this is for example satisfied when $P$ is atomic orbital, see \cite{CLL_Global}*{Lemma~4.3.2}. Then we obtain for every object $A \in P$ a family ${\mathcal P}_A \subseteq P_{/A}$ consisting of the non-terminal objects, that is, those morphisms $p\colon B \to A$ in $P$ which are not equivalences. Specializing to the $P=\Orb$ as in \Cref{ex:FamilyOfSubgroups}, this corresponds to the family of proper subgroups of a finite group $G$.
	\end{example}

	\begin{example}
		\label{ex:PFamilyFf}
		Assume that $P$ is a wide subcategory of some other small $\infty$-category $T$, and let $f\colon A \to B$ be a morphism in $T$. Then there is a $P$-family $\Ff_f$ over $A$ containing those morphisms $p\colon C \to A$ in $P$ for which the composite $fp\colon C \to B$ is also in $P$.
	\end{example}

	\begin{definition}
		Let $\Ff \subseteq P_{/A}$ be a $P$-family over $A$. We define $\EF \in \PSh(P_{/A})$ as the unique presheaf with
		\[
		\EF(B) = \begin{cases} 1 & B \in \Ff \\ \emptyset & B \notin \Ff,  \end{cases}
		\]
		i.e.~$E\mathcal F$ is the composite $P_{/A}\catop \to [1] \hookrightarrow \Spc$,
      where the first functor is the map classifying the sieve $\Ff$ and the second one picks out the morphism of spaces $\emptyset\to1$.

      Under the canonical identification $\PSh(P_{/A}) \xrightarrow{\simeq} \PSh(P)_{/A}$, we will generally think of $\EF$ as a presheaf on $P$ equipped with a morphism $i\colon \EF \hookrightarrow A$. Note that $i$ is a monomorphism, as $\EF \in \PSh(P_{/A})$ is $(-1)$-truncated.
	\end{definition}

	\begin{remark}\label{rk:EF-colim}
		By Kan's pointwise formula \cite{cisinski-book}*{Proposition~6.4.9${}^\op$} and the sieve property, $E\mathcal F$ is the image of the terminal object under the left Kan extension $\PSh(\mathcal F)\to\PSh(P_{/A})$. Writing the terminal object in $\PSh(\mathcal F)$ as a colimit of representables as usual, we therefore get an equivalence
		\[
			E\mathcal F\simeq\mathop{\text{colim}}\limits_{p\in\Ff}\Hom_{P_{/A}}(-,p)=\mathop{\text{colim}}\limits_{p\in\Ff} p_!(1)
		\]
		in $\PSh(P_{/A})$. Applying the equivalence $\PSh(P_{/A})\simeq\PSh(P)_{/A}$, we see that the map $i\colon E\Ff\to A$ of $P$-presheaves is just the tautological map
		\[
			\mathop{\text{colim}}\limits_{(p\colon B\to A)\in\Ff} B\to A.
		\]

	\end{remark}

	\begin{example}
		If $\Ff = P_{/A}$ is the maximal family, we get $\EF = A$.
	\end{example}

    \begin{example}\label{ex:classifying-subgroup}
        Let $\Ff \subseteq \Orb_{G} = \Orb_{/G}$ be a family of subgroups of a finite group $G$. Then $\EF$ agrees (up to the equivalence between $G$-spaces and $\Orb_G$-presheaves provided by Elmendorf's theorem) with the usual universal $G$-space for the family $\Ff$ as recalled in the introduction of this section.
    \end{example}

	Given a presentable $P$-category $\Cc$ and a family $\Ff$ over $A \in P$, we may now define the notion of $\Ff$-torsion and $\Ff$-complete objects.

    \begin{construction}
		Recall from \cite{martiniwolf2022presentable}*{Proposition~8.2.9} that $\Cc$ is canonically tensored over the $P$-category $\ul{\Spc}_P$ of $P$-spaces. More precisely, there is a unique $P$-functor
		\begin{equation}\label{eq:tensoring}
			- \otimes -\colon \Cc \times \;\ul{\Spc}_P \to \Cc
		\end{equation}
		such that $- \otimes 1\colon \Cc \to \Cc$ is the identity and which preserves $P$-colimits in each variable separately in the following sense: for each $A\in P$ the functor
		\[
			- \otimes_A -\colon \Cc(A) \times \PSh(P)_{/A} \to \Cc(A)
		\]
		obtained by evaluating $(\ref{eq:tensoring})$ in degree $A$ preserves (non-parametrized) colimits in both variables, and $\otimes$ satisfies the \emph{projection formula}, i.e.~for every $f\colon A\to B$ in $P$ the canonical maps
		\begin{equation*}
			f_!(-\otimes_A f^*(-))\Rightarrow f_!(-)\otimes_B-
			\qquad\text{and}\qquad
			f_!(f^*(-)\otimes_A -)\Rightarrow -\otimes_B f_!(-)
		\end{equation*}
		obtained as mates of the structure equivalence $f^*(-\otimes_B-)\simeq f^*(-)\otimes_A f^*(-)$ are equivalences.

		In a similar way, we obtain a cotensoring
		\[
			(-)^{(-)} \colon \Cc(A) \times \PSh(P)_{/A}\catop \to \Cc(A)
		\]
    	which preserves limits in both variables.
	\end{construction}

	\begin{example}\label{ex:tensoring-g-spectra}
		Recall the $\Orb$-category $\ul\S$ of equivariant spaces from Example~\ref{ex:Orb-spaces-G-SSet}, giving an explicit model of $\ul\Spc_{\Orb}$.

		We will now make the tensoring of the underlying $\Orb$-category of $\ul\mySp$ explicit in this model: by the uniqueness statement recalled in the previous construction, it will suffice to construct some $\Orb$-functor $\ul{\mySp}\times\ul\S\to\ul\mySp$ that preserves colimits in each variable separately and that restricts to the identity when we plug in the terminal object in the second variable. We claim that
		\begin{equation}\label{eq:tensoring-model}
		\begin{aligned}
		\Fun(-,\text{Sp}^\Sigma)^\text{cof}\times\Fun(-,\text{SSet})&\to \Fun(-,\text{Sp}^\Sigma)^\text{cof}\\
		X,K&\mapsto X\smashp K_+
		\end{aligned}
		\end{equation}
 		descends to the required functor. In particular for each fixed $G$, the parametrized tensoring agrees with the usual tensoring of $G$-spectra by $G$-spaces.

		To verify the claim, we first observe that \cite{hausmann-equivariant}*{Proposition~2.29 and Lemma~4.5} show that for every fixed $G$ the smash product is a left Quillen bifunctor with respect to the usual equivariant model structure on $\Fun(BG,\text{SSet})$ from \cite{cellular}*{Proposition~2.16} (whose weak equivalences are the $G$-equivariant weak equivalences and whose cofibrations are the levelwise injections). This shows that $(\ref{eq:tensoring-model})$ is well-defined and homotopical for every fixed $G$ (as every $G$-simplicial set is cofibrant), and that the induced functor $\mySp_G\times\S_G\to\mySp_G$ preserves colimits in each variable.

		For every \emph{injective} homomorphism $f\colon H\to G$ the left Kan extension functors $f_!\colon\Fun(BH,\text{SSet})\to \Fun(BG,\text{SSet})$ and $f_!\colon\Fun(BH,\text{Sp}^\Sigma)\to \Fun(BG,\text{Sp}^\Sigma)$ are left Quillen (see \cite{g-global}*{Proposition~1.1.18} and \cite{hausmann-equivariant}*{Section~5.2}, respectively), and so the projection formula can be checked on the level of models again, where this is a trivial computation.

		Altogether, we see that the induced functor $\ul\mySp\times\ul\S\to\ul\mySp$ preserves colimits in each variable. As $(\ref{eq:tensoring-model})$ restricts to the identity for the terminal $G$-simplicial sets, this then completes the proof of the claim.

		As an upshot we can in particular simply compute the above abstract tensoring of $\mySp_G$ in terms of the levelwise smash product---in fact, for fixed $G$ there isn't even any need to derive as smashing with a fixed $G$-simplicial set also clearly preserves \emph{level} weak equivalences.
	\end{example}

	\begin{example}\label{ex:global-otimes}
		In the same way, one can identify the above tensoring for global spectra as well as equivariant or global $\Gamma$-spaces as the functor induced by the usual smash product.
	\end{example}

	\begin{definition}[$\Ff$-torsion and $\Ff$-complete objects]
		\label{def:FBorelObjects}
		Let $\Cc$ be a $P$-presentable $P$-category, let $A \in P$, and let $\Ff \subseteq P_{/A}$ be a $P$-family over $A$.
        \begin{itemize}
            \item An object $X \in \Cc(A)$ is called \textit{$\Ff$-torsion} if the canonical map $X \otimes_A \EF \to X \otimes_A A = X$ induced by the map $i\colon \EF \hookrightarrow A$ is an equivalence.
            \item Dually we say that $X$ is \textit{$\Ff$-complete} if the map $X \to X^{\EF}$ is an equivalence.
        \end{itemize}
        We denote by
		\[
		\Cc(A)_{\tors{\Ff}} \subseteq \Cc(A) \quad \text{ and } \quad \Cc(A)_{\comp{\Ff}} \subseteq \Cc(A)
		\]
        the full subcategories of $\Ff$-torsion and $\Ff$-complete objects, respectively.
	\end{definition}

	\begin{example}\label{ex:torsion-equiv}
		Let $P = \Orb$ and let $\Cc = \ul{\mySp}\vert_{\Orb}$ be the $\Orb$-category of genuine equivariant spectra. Combining Examples~\ref{ex:classifying-subgroup} and~\ref{ex:tensoring-g-spectra}, we see that for a family $\Ff$ of subgroups of a finite group $G$ and a genuine $G$-spectrum $X \in \mySp_G$, the map $X \otimes \EF \to X$ from the previous definition agrees with the familiar one, showing that it recovers the usual notion of $\Ff$-torsion recalled in the beginning of this section. Passing to adjoints, we then deduce the same statement for the $\mathcal F$-complete objects.

        The $\Ff$-torsion $G$-spectra admit a well-known interpretation in terms of \emph{geometric fixed points}: $X$ is $\Ff$-torsion if and only $\Phi^H(X)=0$ for all $H\notin\Ff$. Indeed, as the geometric fixed point functors $\Phi^H\colon \mySp_G \to \mySp$ are symmetric monoidal and detect equivalences, this is immediate from the definition of $\EF$.
	\end{example}

	\begin{example}\label{ex:torsion-global}
		Let us describe torsion objects in the underlying $\Orb$-category of the global category $\ul\mySp^\text{gl}$ from Example~\ref{ex:global-spectra}: if $\Ff$ is any family of subgroups of some finite group $G$, then Example~\ref{ex:global-otimes} shows that a $G$-global spectrum $X$ (represented as a symmetric spectrum with $G$-action) is $\Ff$-torsion if and only if the collapse map $X\smashp E\Ff_+\to X$ is a $G$-global stable weak equivalence, i.e.~if and only if $\phi^*(X)\smashp\phi^*(E\Ff)_+\to\phi^*(X)$ is an $H$-stable weak equivalence for every $\phi\colon H\to G$. As $\phi^*(E\Ff)$ is a universal space for the family $\phi^*\Ff\coloneqq\{K\leqslant H\mid\phi(K)\in\Ff\}$, we conclude that a $G$-global spectrum $X$ is $\Ff$-torsion if and only if the underlying $H$-equivariant spectrum of $\phi^*X$ is $\phi^*\Ff$-torsion for every $\phi\colon H\to G$.

		Let us now define $\Phi^\phi X\coloneqq\Phi^H(\phi^*X)$ as the $H$-equivariant geometric fixed points of the $H$-spectrum underlying $\phi^*X$. As usual geometric fixed points are compatible with restriction to subgroups, we then immediately conclude from the above characterization of torsion $G$-spectra that the $G$-global spectrum $X$ is $\Ff$-torsion if and only if $\Phi^{\phi} X = 0$ for all $\phi\colon H\rightarrow G$ with $\textup{im}(\phi)\notin \Ff$.
	\end{example}

        We have the following useful description of the $\Ff$-torsion and $\Ff$-complete objects:

	\begin{lemma}
		\label{lem:BorelObjectsOverBF}
		Let $\Cc$, $A$, and $\Ff$ be as in \Cref{def:FBorelObjects}. Then the restriction functor $i^*\colon \Cc(A) \to \Cc(\EF)$ along the map $i\colon \EF \hookrightarrow A$ in $\PSh(P)$ admits fully faithful left and right adjoints $i_!,i_*\colon \Cc(\EF) \hookrightarrow \Cc(A)$ whose essential images are $\Cc(A)_{\tors{\Ff}}$ and $\Cc(A)_{\comp{\Ff}}$, respectively.
	\end{lemma}

	In particular, while $\Cc(A)_{\tors{\Ff}}$ and $\Cc(A)_{\comp{\Ff}}$ are typically distinct subcategories of $\Cc(A)$, they are abstractly equivalent, both being equivalent to the category $\Cc(E\Ff)$.

	\begin{proof}
        By $P$-presentability of $\Cc$, $i^*$ has a left adjoint $i_!$, and the Beck-Chevalley condition for the pullback diagram
		\[
		\begin{tikzcd}
			\EF \dar[equal] \rar[equal] \drar[pullback] & \EF \dar[hookrightarrow]{i} \\
			\EF \rar[hookrightarrow]{i} & A
		\end{tikzcd}
		\]
        in $\PSh(P)$ shows that the unit $\id\to i^*i_!$ is an equivalence. Thus, \cite{kerodon}*{Tag \texttt{02EX}} implies that $i_!$ is fully faithful with essential image those $X \in \Cc(A)$ for which the \emph{counit} map $i_!i^*X \to X$ is an equivalence. But since the tensoring by $P$-spaces on $\Cc$ preserves $P$-colimits in both variables, the counit is conjugate to the map $X \otimes_A i_!i^*A \to X \otimes_A A = X$ induced by the counit $\EF = i_!i^*A \to A$ in $\PSh(P)_{/A}$, showing that the essential image of $i_!$ is $\Cc(A)_{\tors{\Ff}}$.

		The argument for $i_*$ is dual.
	\end{proof}

	\begin{corollary}\label{cor:torsion-limit}
		In the above situation, the map $\Cc(A)\to\lim_{(p\colon B\to A)\in\Ff^\op}\Cc(B)$
		induced by the restrictions $p^*\colon\Cc(B)\to\Cc(A)$ restricts to equivalences
		\begin{equation*}
			\Cc(A)_{\tors{\Ff}}\simeq\lim_{B\in\Ff^\op} \Cc(B)
			\qquad\text{and}\qquad
			\Cc(A)_{\comp{\Ff}}\simeq\lim_{B\in\Ff^\op} \Cc(B).
		\end{equation*}
	\end{corollary}

	For the underlying $\Orb$-category of $\ul\mySp$, this result appears as \cite{MNN2017Nilpotence}*{Theorem~6.27}, expressing the category of $\Ff$-complete $G$-spectra as a limit of $H$-spectra. On the other hand, specializing it to the underlying $\Orb$-category of $\ul\mySp^\text{gl}$ we get a $G$-global refinement of this result in the form of equivalences
	\begin{equation*}
		(\mySp^\text{gl}_G)_{\tors\Ff}\simeq\lim_{H\in\Ff^\op}\mySp^\text{gl}_H
		\qquad\text{and}\qquad
		(\mySp^\text{gl}_G)_{\comp\Ff}\simeq\lim_{H\in\Ff^\op}\mySp^\text{gl}_H.
	\end{equation*}

	\begin{proof}[Proof of Corollary~\ref{cor:torsion-limit}]
		This follows immediately by combining the lemma with the descent equivalence $\Cc(E\Ff)\simeq\lim_{B\in\Ff^\op}\Cc(B)$ induced by the equivalence $E\Ff\simeq\colim_{B\in\Ff} B$ from Remark~\ref{rk:EF-colim}.
	\end{proof}

	\begin{remark}\label{rk:F-we}
		The lemma in particular tells us that $i^*$ is a localization in the above situation. More generally, for any $P$-category $\Cc$, we will refer to maps $f$ in $\Cc(A)$ for which $i^*f$ is an equivalence as \emph{$\Ff$-weak equivalences}. Using descent once more, we see this is equivalent to $p^*f$ being an equivalence for every $p\in\Ff$.
	\end{remark}

	As another application of the lemma we obtain the following decomposition result:

	\begin{corollary}
		\label{cor:BorelObjectsGenerated}
		The subcategory $\Cc(A)_{\tors{\Ff}} \subseteq \Cc(A)$ is closed under colimits and contains all objects of the form $p_!Y$ for $(p\colon C\to A)\in\mathcal F$ and $Y\in\Cc(C)$. Conversely, every $\Ff$-torsion object $X$ in $\Cc(A)$ can be written as a colimit of objects of this form.
	\end{corollary}
	\begin{proof}
            The first two claims follow from the identification of $\Cc(A)_{\tors{\Ff}}$ with the essential image of $i_!$, as $i_!$ preserves colimits and every morphism $p\colon C \to A$ in $\Ff$ factors through $i\colon E\mathcal F\to A$ (see Remark~\ref{rk:EF-colim}), so that $p_!$ factors through $i_!$. For the final claim, we write $E\Ff\simeq\colim_{(p\colon C\to A)\in\mathcal F}p_!C$ as in Remark~\ref{rk:EF-colim} and consider for every $\Ff$-torsion object $X$ the composite equivalence
		\begin{equation*}
			X\simeq X\otimes_A E\mathcal F\simeq\mathop{\text{colim}}\limits_{p\in\mathcal F}X\otimes_Ap_!C\simeq\mathop{\text{colim}}\limits_{p\in\mathcal F} p_!(p^*X\otimes_CC)\simeq\mathop{\text{colim}}\limits_{p\in\mathcal F} p_!p^*X,
		\end{equation*}
		where the penultimate equivalence uses the projection formula.
	\end{proof}

	\section{The abstract Adams isomorphism}\label{sec:adams-iso-gen}
	Throughout this section, we fix a small category $T$ and an atomic orbital subcategory $P\subseteq T$. The goal of this section is to prove an abstract Adams isomorphism in a $P$-presentable $P$-semiadditive $T$-category $\Cc$: under suitable assumptions on a map $f\colon A \to B$ in $T$, we will construct a map $\Nm_f\colon f_!^{\fr}(X) \to f_*(X)$ from the \textit{partially defined} left adjoint $f_!^{\fr}$ of $f^*$ to its right adjoint $f_*$, and show that it is an equivalence. One subtlety here is that the functor $f^*\colon \Cc(B) \to \Cc(A)$ need not have a fully defined left adjoint. Instead, we will show that it at least exists on a certain class of \textit{$f$-free objects}:

	\begin{definition}
		\label{def:fBorelObjects}
		For a morphism $f\colon A \to B$ in $T$, recall from \Cref{ex:PFamilyFf} the $P$-family $\Ff_f \subseteq P_{/A}$ consisting of those maps $p\colon C \to A$ in $P$ such that also the composite $fp\colon C \to B$ is in $P$. We will denote by
		\[
		      \Cc(A)_{\free{f}} \subseteq \Cc(A)
		\]
            the full subcategory of $\Ff_f$-torsion objects, in the sense of \Cref{def:FBorelObjects} applied to the underlying $P$-category, and refer to its objects as \textit{$f$-free objects}.
	\end{definition}

	\begin{example}\label{ex:P-free}
		If $f$ is itself a morphism in $P$, then $\Ff_f = P_{/A}$ and thus every object $X \in \Cc(A)$ is $f$-free.
	\end{example}

        \begin{example}
            Let $T = \Glo$, $P = \Orb$, and let $\Cc = \ul{\mySp}$ be the global category of equivariant spectra from Example~\ref{ex:equiv-spectra}. If $f$ is the projection $G \to G/N$ for a normal subgroup $N \trianglelefteqslant G$, then Example~\ref{ex:torsion-equiv} shows that a $G$-spectrum $X \in \mySp_G$ is $f$-free if and only if it is $N$-free in the sense of the introduction.
        \end{example}

	We now construct the promised partial left adjoint:

	\begin{lemma}
		\label{lem:PartialLeftAdjoint}
		Let $\Cc$ be $P$-presentable and let $f\colon A \to B$ be a map in $T$. There exists a unique functor $f_!^{\fr}\colon \Cc(A)_{\free{f}} \to \Cc(B)$ equipped with a natural equivalence
		\begin{equation}\label{eq:def-f!}
		\Hom_{\Cc(A)}(X,f^*Y) \simeq \Hom_{\Cc(B)}(f_!^{\fr}X,Y)
		\end{equation}
		of mapping spaces for $X \in \Cc(A)_{\free{f}}$ and $Y \in \Cc(B)$.
	\end{lemma}
	\begin{proof}
		It suffices to show that the left hand side of $(\ref{eq:def-f!})$ is corepresentable for every $f$-free $X$. As corepresentable functors are closed under limits, we may assume by Corollary~\ref{cor:BorelObjectsGenerated} that $X=p_!X'$ for some $(p\colon C\to A)\in\Ff_f$ and $X'\in\Cc(C)$, in which case $\Hom(X,f^*Y)\simeq\Hom(X',p^*f^*Y)\simeq\Hom(X',(fp)^*Y)$. By definition of $\Ff_f$, the map $fp\colon C\to B$ lies in $P$, so $(fp)_!X$ is the desired corepresenting object.
	\end{proof}

	We now come to the comparison map $f_!^{\fr}\to f_*$, which is a slight modification of the norm construction of Hopkins and Lurie \cite{hopkins2013ambidexterity}*{Construction~4.1.8}. One obstacle we will encounter here is that the right adjoints do not satisfy base change in general (cf.~Remark~\ref{rk:basechange}). However, we at least have the following weaker statement:

	\begin{lemma}
        \label{lem:Beck_Chevalley_Ff_Weak_Equivalence}
		Let $\Cc$ be a $P$-presentable $T$-category and let
		\begin{equation*}
			\begin{tikzcd}
				X\arrow[d, "g'"']\arrow[dr,phantom,"\lrcorner"{very near start}]\arrow[r, "f'"] & Y\arrow[d,"g"]\\
				A\arrow[r, "f"'] & B
			\end{tikzcd}
		\end{equation*}
		be a pullback in $\PSh(T)$ such that $f$ is a map in $T$. Then the Beck-Chevalley map $f^*g_*\to g'_*(f')^*$ is an $\mathcal F_f$-weak equivalence (cf.~Remark~\ref{rk:F-we}).
		\begin{proof}
			Let $p\colon C\to A$ be a map in $P$ such that $fp$ belongs to $P$. Then in the iterated pullback
			\begin{equation*}
				\begin{tikzcd}
					C\times_AX\arrow[r]\arrow[d]\arrow[dr,phantom,"\lrcorner"{very near start}]&X\arrow[d, "g'"']\arrow[dr,phantom,"\lrcorner"{very near start}]\arrow[r, "f'"] & Y\arrow[d,"g"]\\
					C\arrow[r, "p"'] & A\arrow[r, "f"'] & B
				\end{tikzcd}
			\end{equation*}
			the Beck-Chevalley maps associated to the left hand square and to the total rectangle are equivalences by $P$-presentability. Thus, $2$–out-of-$3$ implies that the Beck-Chevalley map of the right hand square is inverted by $p^*$.
		\end{proof}
	\end{lemma}

	\begin{construction}
		Let $f\colon A\to B$ be a map in $\PSh(T)$ and consider the pullback
		\begin{equation*}
			\begin{tikzcd}
				A\times_BA\arrow[dr,phantom,"\lrcorner"{very near start}]\arrow[r, "\pr_1"]\arrow[d, "\pr_2"'] & A\arrow[d, "f"]\\
				A\arrow[r, "f"'] & B\rlap.
			\end{tikzcd}
		\end{equation*}
		Assume that the diagonal $\Delta\colon A\to A\times_BA$ is a map in $\ul{\mathbb F}^P_T$. Let $\Cc$ be $P$-semiadditive, and assume that the required right adjoints exist in the following zig-zag:
            \begin{equation}\label{eq:Nm-zig-zag} \id\simeq\pr_{1*}\Delta_*\Delta^*\pr_1^*\xrightarrow[\raise2.5pt\hbox{$\scriptstyle\sim$}]{\Nm_\Delta^{-1}}\pr_{1*}\Delta_!\Delta^*\pr_1^*\xrightarrow{\,\epsilon\,}\pr_{1*}\pr_2^*\xleftarrow{\text{BC}_*} f^*f_*.
		\end{equation}
        \begin{enumerate}[(a)]
            \item We write $N\colon \id\to\pr_{1*}\pr_2^*$ for the composite of the first three maps in \eqref{eq:Nm-zig-zag}.
            \item If $\Cc$ is in addition $P$-presentable, then \Cref{lem:Beck_Chevalley_Ff_Weak_Equivalence} shows that the Beck-Chevalley map $\BC_*\colon f^*f_*\to \pr_{1*}\pr_2^*$ becomes an equivalence after applying $i^*\colon \Cc(A) \to \Cc(E\Ff_f)$. Applying $i^*$ to \eqref{eq:Nm-zig-zag} thus gives a map
            \[
                \smash{\Nmadjdual}^{\fr}_f\colon i^* \to i^*f^*f_*.
            \]
            \item By restricting $\smash{\Nmadjdual}^{\fr}_f$ along $i_!$ and passing to adjoints, we obtain a map $(fi)_! \to f_*i_!$ of functors $\Cc(E\Ff_f) \to \Cc(B)$, which we may equivalently view as a natural transformation
            \[
                \Nm_f^{\fr}\colon f_!^{\fr} \to f_*
            \]
            of functors $\Cc(A)_{\free{f}} \to \Cc(B)$.
        \end{enumerate}
	\end{construction}

	\begin{remark}
		If $\Cc$ is $P$-semiadditive and $T$-presentable, the Beck-Chevalley map $f^*f_*\to \pr_{1*}\pr_2^*$ is already an equivalence before applying $i^*$, and so \eqref{eq:Nm-zig-zag} determines a map $\overline{\Nm}_f\colon \id_{\Cc(A)}\to f^*f_*$ with $i^*\overline{\Nm}_f=\smash{\Nmadjdual}^{\fr}_f$. Moreover, $f^*$ admits a globally defined left adjoint $f_!$, and $\Nmadjdual_f$ adjoints to a map $\Nm_f\colon f_!\to f_*$ restricting to $\Nm^{\fr}_f$ on $\Cc(A)_{\free{f}}$. The maps $\smash{\Nmadjdual}_f$ and $\Nm_f$ are then literally instances of (the dual of) the norm construction of Hopkins and Lurie \cite{hopkins2013ambidexterity}*{Construction~4.1.8}.
	\end{remark}

	We are now ready to state our main result:

	\begin{theorem}[Abstract Adams isomorphism]\label{thm:adams-gen}
		Let $\Cc$ be a $P$-semiadditive $P$-presentable $T$-category and let $f\colon A\to B$ be a map in $T$ such that the diagonal $A\to A\times_BA$ lies in $\ul{\mathbb F}^P_T$. Then:
		\begin{enumerate}
            \item \label{item:ag-generators}For any $p\colon C\to A$ in $P$ such that also $fp$ is a map in $P$ and every $X\in\Cc(A)$ in the essential image of $p_!$, the map $\Nm_f^{\fr}\colon f^{\fr}_! X \to f_* X$ is an equivalence.
			\item If $f_*$ preserves colimits, then $\Nm_f^{\fr}\colon f^{\fr}_!X\to f_*X$ is an equivalence for every $f$-free object $X$.
		\end{enumerate}
	\end{theorem}

	While the proof will be given below after some preparations, let us already note the following consequence in the stable setting:

	\begin{corollary}\label{cor:Adams-stable-compact}
		Let $\Cc$ be $P$-stable and $P$-presentable, and let $f\colon A\to B$ be a map in $T$ whose diagonal belongs to $\ul{\mathbb F}^P_T$. Then $\Nm_f^{\fr}\colon f^{\fr}_!X\to f_*X$ is an equivalence for every \emph{compact} $f$-free object $X$.
		\begin{proof}
			By Corollary~\ref{cor:BorelObjectsGenerated}, any $X\in\Cc(A)_{\free{f}}$ can be written as $\colim_{k\in K} p_{k!}X_k$ for some category $K$, maps $p_k\colon A_k\to A$ in $P$ with $fp_k$ in $P$, and objects $X_k\in\Cc(A_k)$. Filtering $K$ by finite categories, we can hence write $X$ as a filtered colimit of objects of the form $\colim_{k\in K} p_{k!}X_k$ with $p_k, X_k$ as above and such that $K$ is in addition \emph{finite}. If $X$ is now also compact, it is therefore a retract of such a finite colimit.

			By the theorem, $\Nm_f^{\fr}\colon f_!(p_{k!}X_k)\to f_*(p_{k!}X_k)$ is an equivalence for every $k$. The claim follows as $f_!$ and $f_*$ preserve retracts and \emph{finite} colimits (by stability).
		\end{proof}
	\end{corollary}

	We now turn to the proof of Theorem~\ref{thm:adams-gen}. We begin with a $T$-presentable version of the first statement:

	\begin{lemma}\label{lemma:Adams-T-complete}
		Let $\Cc$ be $T$-presentable and $P$-semiadditive, let $p\colon C\to A$ be a map in $P$, and let $f\colon A\to B$ be a map in $T$ such that also $fp$ belongs to $P$. Then ${\Nm}_f\colon f_!X=f^{\fr}_!X\to f_*X$ is an equivalence whenever $X \in \Cc(A)$ lies in the essential image of $p_!\colon \Cc(C) \to \Cc(A)$.
		\begin{proof}
         	Writing $X = p_!Y$ for some $Y\in\Cc(C)$, \cite{hopkins2013ambidexterity}*{Remark~4.2.4${}^\op$} shows that
			\begin{equation*}
				(fp)_!Y\simeq f_!p_!Y\xrightarrow{\Nm_f} f_*p_!Y\xrightarrow{\Nm_p} f_*p_*Y\simeq (fp)_*Y
			\end{equation*}
			agrees with ${\Nm}_{fp}$. As both $\Nm_{fp}$ and $\Nm_p$ are equivalences by $P$-semiadditivity, the claim follows by $2$-out-of-$3$.
		\end{proof}
	\end{lemma}

	We will now deduce the theorem by reducing to the above lemma via an embedding trick.

	\begin{lemma}\label{lemma:embed}
		Let $\Cc$ be a small $P$-semiadditive $T$-category. Then there exists a $T$-presentable $P$-semiadditive $T$-category $\Dd$ together with a fully faithful $P$-semiadditive functor $\Cc\to\Dd$.
		\begin{proof}
			The Yoneda embedding $y\colon\Cc\to\ul\Fun_T(\Cc^\op,\ul\Spc_T)$ of \cite{martini2021yoneda} is fully faithful and preserves finite $P$-products, see \cite{martini2021yoneda}*{Corollary~4.7.16} and \cite{martiniwolf2021limits}*{Proposition~4.4.8}. We now take $P$-commutative monoids in the sense of \cite{CLL_Global}*{Definition~4.8.1} on both sides, yielding a $P$-semiadditive functor
			\begin{equation}\label{eq:emb-cmon}
			\ul\CMon^P(y)\colon\ul\CMon^P(\Cc)\to\ul\CMon^P\big(\ul\Fun_T(\Cc^\op,\ul\Spc_T)\big)
			\end{equation}
			of $P$-semiadditive $T$-categories. As $\ul\CMon^P(-)$ is defined as a full subcategory of a functor category, $\ul\CMon^P(y)$ is still fully faithful by \cite{martini2021yoneda}*{Proposition~3.8.4}.

			Now \cite{CLL_Global}*{Proposition~4.8.3} shows that the source of $(\ref{eq:emb-cmon})$ is equivalent to just $\Cc$ itself while Proposition~4.6.17 of \emph{op.\ cit.}~shows that its target is $T$-presentable. The composite $\Cc\simeq\ul\CMon^P(\Cc)\to\ul\CMon^P\big(\ul\Fun_T(\Cc^\op,\ul\Spc_T)\big)$ therefore has the required properties.
		\end{proof}
	\end{lemma}

	To apply this, we need to understand how the right adjoint $f_*$ and our newly constructed norm map interact with this embedding:

	\begin{lemma}\label{lemma:BC-inv}
		Let $\Cc,\Dd$ be $T$-categories with finite $P$-products and let $F\colon\Cc\to\Dd$ preserve finite $P$-products. Let $f\colon B\to C$ be a map in $T$ such that both $f^*\colon\Cc(C)\to\Cc(B)$ and $f^*\colon\Dd(C)\to\Dd(B)$ admit right adjoints $f_*$, and let $p\colon A\to B$ be a map in $P$ such that also $fp$ belongs to $P$.

		Then the Beck-Chevalley map $Ff_*X\to f_*FX$ is an equivalence whenever $X \in \Cc(A)$ lies in the essential image of $p_*\colon \Cc(C) \to \Cc(A)$.
		\begin{proof}
			We may assume without loss of generality that $X=p_*Y$ with $Y\in\Cc(A)$. By the compatibility of Beck-Chevalley maps with composition (Lemma~\ref{lemma:BC-compose}${}^\op$), the composite $Ff_*p_*Y \to f_*Fp_*Y \to f_*p_*FY$ of Beck-Chevalley maps then agrees with the Beck-Chevalley map $Ff_*p_*Y\to f_*p_*FY$, which we can further identify up to conjugation by the obvious equivalences with the Beck-Chevalley map $F(fp)_*Y\to (fp)_*F$ (see Corollary~\ref{cor:BC-htpy-inv}${}^\op$). As $p$ and $fp$ are maps in $P$ and $F$ preserves finite $P$-products, the claim now simply follows by $2$-out-of-$3$.
		\end{proof}
	\end{lemma}

	\begin{proposition}\label{prop:nm-lambda-preserved}
		Let $F\colon\Cc\to\Dd$ be a $P$-semiadditive functor of $P$-semiadditive $T$-categories. Let $f\colon A\to B$ be a map in $T$ such that the diagonal $\Delta\colon A\to A\times_BA$ lies in $\ul{\mathbb F}^P_T$ and such that all the requisite adjoints exist to define $\smash{\Nmadjdual}^{\fr}_f$ for $\Cc$ and $\Dd$. Then the diagrams
		\begin{equation*}
			\begin{tikzcd}
				Ff^*f_* \arrow[d, "\simeq"']\arrow[r, "F(\textup{BC}_*)"] &[1em] F\pr_{1*}\pr_2^*\arrow[d, "\textup{BC}_*"]\\
				f^*Ff_*\arrow[d, "\textup{BC}_*"'] & \pr_{1*}F\pr_2^*\arrow[d, "\simeq"]\\
				f^*f_*F\arrow[r, "\textup{BC}_*"'] &  \pr_{1*}\pr_2^*F
			\end{tikzcd}
			\qquad\text{and}\qquad
			\begin{tikzcd}
				F\arrow[r, "F(N)"]\arrow[d, "N"'] & F\pr_{1*}\pr_2^*\arrow[d,"\textup{BC}_*"]\\
				\pr_{1*}\pr_2^*F \arrow[r, "\simeq"'] &\pr_{1*}F\pr_2^*
			\end{tikzcd}
		\end{equation*}
		commute up to homotopy.
		\begin{proof}
			For the first diagram we observe that by the compatibility of Beck-Chevalley maps with composition (Lemma~\ref{lemma:BC-compose}${}^\op$) the subcomposites $Ff^*f_*\to \pr_{1*}F\pr_{2^*}$ and $f^*Ff_*\to \pr_{1*}\pr_2^*F$ are given by the Beck-Chevalley transformations associated to the total rectangles in the two diagrams
			\begin{equation*}
				\begin{tikzcd}[cramped]
					\Cc(B)\arrow[r, "f^*"]\arrow[d, "f^*"'] &[-.5em] \Cc (A)\arrow[d, "\pr_2^*"{description}] \arrow[r, "F"] &[-.5em] \Dd(A)\arrow[d, "\pr_2^*"]\\
					\Cc(A)\arrow[r, "\pr_2^*"'] & \Cc(A\times_BA) \arrow[r, "F"'] & \Dd(A\times_BA)
				\end{tikzcd}
				\qquad\quad
				\begin{tikzcd}[cramped]
					\Cc(B)\arrow[r, "F"]\arrow[d, "f^*"'] &[-.5em] \Dd (B)\arrow[d, "f^*"{description}] \arrow[r, "f^*"] &[-.5em] \Dd(A)\arrow[d, "\pr_2^*"]\\
					\Cc(A)\arrow[r, "F"'] & \Dd(A) \arrow[r, "\pr_2^*"'] & \Dd(A\times_BA)\rlap.
				\end{tikzcd}
			\end{equation*}
			As the total rectangles only differ up to the naturality equivalences $Ff^*\simeq f^*F$ and $F\pr_2^*\simeq\pr_2^*F$, the first statement now follows from Corollary~\ref{cor:BC-htpy-inv-vert}${}^\op$.

			For the second statement we consider the diagram
			\begin{equation*}\hskip-8.94pt\hfuzz=9pt
				\begin{tikzcd}[cramped]
					F\arrow[r, "\simeq"]\arrow[dd, "\simeq"']\arrow[ddr,phantom,"(*)"{description,xshift=-9pt}] & F\pr_{1*}\Delta_*\Delta^*\pr_{2}^*\arrow[d, "\text{BC}_*"]\arrow[r, "\Nm_\Delta^{-1}"] &[1em] F\pr_{1*}\Delta_!\Delta^*\pr_{2}^*\arrow[r, "\epsilon"]\arrow[d, "\text{BC}_*"] & F\pr_{1*}\pr_2^*\arrow[d, "\text{BC}_*"] \\
					& \pr_{1*}F\Delta_*\Delta^*\pr_2^*\arrow[r, "\Nm_\Delta^{-1}"{description}]\arrow[d, "\text{BC}_*"]\arrow[dr,phantom, "(\dagger)"{description, yshift=-1pt}] & \pr_{1*}F\Delta_!\Delta^*\pr_2^*\arrow[r, "\epsilon"] & \pr_{1*}F\pr_2^*\\
					\pr_{1*}\Delta_*\Delta^*\pr_2^*F\arrow[r, "\simeq"]&\pr_{1*}\Delta_* F \Delta^*\pr_2^*\arrow[r, "\Nm_\Delta^{-1}"'] & \pr_{1*}\Delta_!F\Delta^*\pr_2^*\arrow[u, "\text{BC}_!"']\arrow[r, "\simeq"'] & \pr_{1*}\Delta_!\Delta^*F\pr_2^*\arrow[u, "\epsilon"']\arrow[ul,phantom, "(\ddag)"{description}]
				\end{tikzcd}
			\end{equation*}
			Here the unmarked squares commute by naturality, while the subdiagram $(\ddag)$ commutes by Lemma~\ref{lemma:BC-unit-counit}. The square $(\dagger)$ commutes before inverting the horizontal equivalences by \cite{CLL_Global}*{Lemma~4.4.8${}^\op$}; since also the vertical maps are equivalences by the semiadditivity assumption on $F$, we conclude that also $(\dagger)$ commutes. To see that also the final remaining subdiagram $(*)$ commutes, we first note that by another application of Lemma~\ref{lemma:BC-compose}${}^\op$ the right hand vertical composite $F\pr_{1*}\Delta_*\Delta^*\pr_2^*\to \pr_{1*}\Delta_*F\Delta^*\pr_2^*$ is induced by the Beck-Chevalley map $F\pr_{1*}\Delta_*\to \pr_{1*}\Delta_*F$ for the adjunction $\Delta^*\pr_{1}^*\dashv\pr_{1*}\Delta_*$. By Corollary~\ref{cor:BC-htpy-inv}${}^\op$ we can therefore rewrite the top right composite as $F\simeq \pr_{1*}\Delta_*F\simeq  \pr_{1*}\Delta_*F\Delta^*\pr_2^*$. As this further agees with the bottom left composite by naturality, this completes the proof that $(*)$ and hence the whole diagram commutes.

			The top row spells out the definition of $F(N)\colon F\to F\pr_{1*}\pr_2^*$; it therefore remains to show that the lower composite $F\to \pr_{1*}F\pr_2^*$ agrees with the composite $F\to\pr_{1*}\pr_2^*F\simeq\pr_{1*}F\pr_2^*$. For this we contemplate the diagram
			\begin{equation*}
				\begin{tikzcd}
					F\arrow[d,"\simeq"']\\
					\pr_{1*}\Delta_*\Delta^*\pr_2^*F\arrow[d, "\Nm_\Delta^{-1}"'] \arrow[r,"\simeq"] & \pr_ {1*}\Delta_*F\Delta^*\pr_2^*\arrow[d, "\Nm_\Delta^{-1}"]\\
					\pr_{1*}\Delta_!\Delta^*\pr_2^*F\arrow[r,"\simeq"]\arrow[d,"\epsilon"'] & \pr_{1*}\Delta_!F\Delta^*\pr_2^*\arrow[r,"\simeq"] &
					\pr_{1*}\Delta_!\Delta^*F\pr_{2}^*\arrow[d,"\epsilon"]\\
					\pr_{1*}\pr_{2}^*F\arrow[rr,"\simeq"'] && \pr_{1*}F\pr_2^*
				\end{tikzcd}
			\end{equation*}
			where all unlabelled equivalences come from the compatibility of $F$ with restrictions. By the coherence conditions for the latter, the composite of the middle row is then just induced by the equivalence $\pr_2^*F\simeq F\pr_2^*$; with this established, all squares simply commute by naturality. The top right path $F\to\pr_{1*}F\pr_2^*$ literally agrees with the bottom composite in the previous diagram; on the other hand, the left hand column spells out the definition of the map $N\colon F\to\pr_{1*}\pr_2^*F$. Thus, this diagram witnesses the desired identity, completing the proof of the proposition.
		\end{proof}
	\end{proposition}

	\begin{proof}[Proof of Theorem~\ref{thm:adams-gen}]
		We will prove the first statement, the second one will then follow from this via Corollary~\ref{cor:BorelObjectsGenerated}, similarly to the proof of Corollary~\ref{cor:Adams-stable-compact}.

		For this, let $i$ again denote the monomorphism $E\mathcal F_f\to A$. The claim then translates to demanding that the composite
		\begin{equation}\label{eq:norm-as-unit}
			\Hom(f_*X,Y)\xrightarrow{i^*f^*}\Hom(i^*f^*f_*X,i^*f^*Y)\xrightarrow{\blank\circ\smash{\overline\Nm}^{\fr}_f}\Hom(i^*X,i^*f^*Y)
		\end{equation}
		be an equivalence for every $X\in\essim(p_!)$ and $Y\in\Cc(B)$; similarly, Lemma~\ref{lemma:Adams-T-complete} translates to saying that the analogously defined map is an equivalence in every $P$-semiadditive $T$-presentable $\Dd$.

		Passing to a larger universe, we may assume that $\Cc$ is small. We now employ Lemma~\ref{lemma:embed} to obtain a $T$-presentable $P$-semiadditive $\Dd$ together with a fully faithful $P$-semiadditive functor $\Cc\to\Dd$, and we consider the diagram
		\begin{equation*}\hskip-35.18pt\hfuzz=35.2pt
			\begin{tikzcd}[cramped]
				[f_*X,Y]\arrow[r, "i^*f^*"]\arrow[d, "F"'] &\arrow[d, "F"][i^*f^*f_*X,i^*f^*Y]\arrow[r, "\blank\circ i^*(\text{BC}_*)^{-1}"]&[2.5em] \arrow[d, "F"][i^*\pr_{1*}\pr_2^*X,i^*f^*Y] \arrow[r, "\blank\circ N"]& \arrow[d,"F"] [i^*X,i^*f^*Y]\\
				{[Ff_*X,FY]}\arrow[r, "i^*f^*"]& {[i^*Ff^*f_*X,i^*f^*FY]} \arrow[r, "(\blank\circ \text{BC}_*)^{-1}"]\arrow[dr,phantom,"(*)"{description}]  &
				{[i^*F\pr_{1*}\pr_2^*X,i^*f^*FY]} \arrow[dr,phantom,"(*)"{description}]\arrow[r, "\blank\circ N"] & {[i^*FX,i^*f^*FY]}\\
				{[f_*FX,FY]}\arrow[u, "\text{BC}_*", "\sim"']\arrow[r, "i^*f^*"'] & {[i^*f^*f_*FX, i^*f^*FY]}\arrow[u, "\text{BC}_*"]\arrow[r, "\blank\circ i^*(\text{BC}_*)^{-1}\,"'] & {[i^*\pr_{1*}\pr_2^*FX, i^*f^*FY]}\arrow[u, "\text{BC}_*"']\arrow[r, "\blank\circ N"'] & {[i^*FX,i^*f^*FY]}\arrow[u,equal]
			\end{tikzcd}
		\end{equation*}
		where we have written $[\,{,}\,]$ for $\Hom$ and have suppressed the naturality constraints of $F$ for simplicity. Here the top row spells out $(\ref{eq:norm-as-unit})$ while the bottom row spells out the analogously defined map for $\Dd$.

		The two rectangles marked $(*)$ commute by Proposition~\ref{prop:nm-lambda-preserved} while all other rectangles commute by naturality. By $P$-semiadditivity, $F(X)$ lies in the essential image of $p_!\colon\Dd(C)\to\Dd(A)$, so Lemma~\ref{lemma:Adams-T-complete} shows that the bottom composite $[f_*FX,FY]\to[i^*FX,i^*f^*FY]$ is an equivalence.

		On the other hand, Lemma~\ref{lemma:BC-inv} together with $P$-semiadditivity shows that the lower left vertical map $\BC_*\colon [f_*FX,FY]\to[Ff_*X,FY]$ is an equivalence. As $F$ is fully faithful, also all top vertical maps are equivalences, so the claim now follows by $2$-out-of-$3$.
	\end{proof}

	\section{Adams isomorphisms in equivariant mathematics}
    \label{sec:Adams_Iso_Equivariant}
	Our main interest of the abstract Adams isomorphism from \Cref{thm:adams-gen} is in the case where $T = \Glo$ and $P = \Orb$. In this section, we will spell out the resulting Adams isomorphism in this setting, and provide a range of examples.

	\begin{definition}[$N$-free $G$-objects]
		\label{def:NBorelObjects}
		Let $N$ be a normal subgroup of a finite group $G$. Specializing \Cref{ex:PFamilyFf} to the quotient map $f\colon G\to G/N$, we get $E\mathcal F_N\coloneqq E\mathcal F_f\in\ul\Spc_{\Orb}(G)$, which we can equivalently view as the $G$-space satisfying
		\[
		(\EF_N)^H = \begin{cases} 1 & \text{if }H\cap N=e \\ \emptyset & \text{otherwise}.
		\end{cases}
		\]
		For any equivariantly presentable global category $\Cc$, we then define the subcategory $\Cc(G)_{\Nfree} \subseteq \Cc(G)$ of \textit{$N$-free $G$-objects} as the full subcategory of $(f\colon G\to G/N)$-free objects in the sense of \Cref{def:FBorelObjects}, i.e.~the subcategory of those objects $X \in \Cc(G)$ for which the map $X \otimes_G \EF_N\to X$ is an equivalence.
	\end{definition}

	We then obtain the following specialization of the abstract Adams isomorphism:

	\begin{theorem}[Abstract Adams isomorphism for global categories]
		\label{thm:GlobalFormalAdamsIso}
		Let $\Cc\hskip0pt minus .5pt\colon\hskip0pt minus 1pt \Glo\catop\hskip0pt minus 1pt \to \Cat$ be an equivariantly presentable equivariantly semiadditive global category and let $N \trianglelefteqslant G$ be a normal subgroup of a finite group $G$. Assume that the $N$-fixed point functor $(-)^N\colon \Cc(G) \to \Cc(G/N)$ preserves (non-parametrized) colimits. Then there is a natural equivalence
		\[
		X/N \iso X^N
		\]
		in $\Cc(G/N)$ for every $N$-free $G$-object $X \in \Cc(G)_{\Nfree}$.
	\end{theorem}
	\begin{proof}
		This is an instance of \Cref{thm:adams-gen}, applied to the quotient morphism $f\colon G \to G/N$. The diagonal of $f$ is the map $G \to G \times_{G/N} G$, which is injective as it admits a retraction, and thus lies in $\Orb$.
	\end{proof}

	In the remainder of this section, we discuss a wide range of examples of this theorem.

	\begin{example}[Equivariant homotopy theory]
		\label{ex:EquivHtpyTheory}
		Let $\Cc=\ul\mySp$ be the global category of genuine equivariant spectra from Example~\ref{ex:equiv-spectra}. For any finite group $H$, the suspension spectra of transitive $H$-sets form a set of compact generators for $\mySp_H$, see \cite{hausmann-equivariant}*{Proposition~4.9(3)}. Given a quotient map $q\colon G \to G/N$, we immediately see that the inflation functor $q^*\colon \mySp_{G/N} \to \mySp_G$ sends these compact generators to compact objects, so that its right adjoint $(-)^N$ preserves colimits.

		Thus, we may apply the above theorem to $\ul\mySp$. The resulting equivalence $X/N \simeq X^N$ for $N$-free genuine $G$-spectra is precisely the classical Adams isomorphism.
	\end{example}

\begin{example}[Global equivariant homotopy theory] \label{ex:GlobalEquivHtpyTheory}
	Similarly, consider the global category $\ul\mySp^\text{gl}$ of global spectra (Example~\ref{ex:global-spectra}). Given any quotient map $q\colon G\to G/N$, the functor $(-)^N=q_*\colon\mySp^\text{gl}_G\to\mySp^\text{gl}_{G/N}$ is cocontinuous as a special case of Theorem~\ref{thm:universal-cocontinuous} in the appendix. Let us also give a direct argument for this:

	For each $H$, the objects $\phi_!\mathbb S$ for $\phi$ running through all homomorphisms into $H$ form a set of compact generators of $\mySp^\text{gl}_H$, see \cite{CLL_Global}*{Theorem~7.1.12}. Taking $H=G/N$, the Beck-Chevalley condition shows that $q^*\phi_!\mathbb S$ is equivalent to $(G\times_{G/N}\phi)_!\mathbb S$, hence in particular again compact; as before, we then conclude that the right adjoint $(-)^N$ preserves colimits.

	As a consequence of the theorem, we therefore get equivalences of $G/N$-global spectra $X/N\simeq X^N$ for any $N$-free $G$-global spectrum $X$, i.e.~those $X$ such that $\Phi^\phi X=0$ whenever $\textup{im}(\phi)\cap N\not=e$, see \Cref{ex:torsion-global}. In the special case $G=N$ an alternative approach to the Adams isomorphism for $G$-global spectra, based on the usual proof of the equivariant Adams isomorphism, appears in \cite{tigilauri}.
\end{example}

\begin{warn}
	As the global category $\ul\mySp^{\gl}$ is even globally presentable, the restriction $q^*\colon\Cc(G/N)\to\Cc(G)$ has an honest left adjoint ${-}/N$, and the norm construction provides a comparison map $X/N\to X^N$ for every $G$-global spectrum $X$, not just the $N$-free ones. Beware, however, that this map is typically far from being an equivalence. As a concrete example, let $G=N=C_2$ and $X=\mathbb S$ be the $C_2$-global sphere. Then $X/C_2\simeq\Sigma^\infty_+(*/C_2)=\Sigma^\infty_+ B_\text{gl}C_2$ is the suspension spectrum of the global classifying space of $C_2$, so its underlying non-equivariant spectrum is the suspension spectrum of the usual classifying space $BC_2$. On the other hand, the underlying spectrum of $X^{C_2}$ is given by the categorical $C_2$-fixed points of $X$, so the tom Dieck splitting identifies it with $\mathbb S\vee \Sigma^\infty_+BC_2$. Thus, already the zeroth non-equivariant homotopy groups of $X/C_2$ and $X^{C_2}$ differ.
\end{warn}

	\begin{example}\label{ex:gamma-Adams}
		We further recall the global categories $\ul\GammaS_*^\text{spc}$ and $\ul\GammaS_*^\text{gl,\,spc}$ of equivariant and global special $\Gamma$-spaces, respectively (Example~\ref{ex:gamma}). Given any quotient map $q\colon G\to G/N$, Theorem~\ref{thm:universal-cocontinuous} shows that the right adjoints $q_*=(-)^N$ are cocontinuous, so the theorem provides us with Adams isomorphisms for $N$-free $G$-equivariant and $G$-global special $\Gamma$-spaces. In light of \cite{ostermayr}*{Theorem~5.9} and \cite{g-global}*{Theorem~3.4.21}, these can be viewed as a `non-group completed' versions of the Adams isomorphisms discussed in the previous two examples, and in particular provide a further refinement of them in the connective case.
	\end{example}

	\begin{example}[Mackey functors]
		\label{ex:MackeyFunctors}
		Let $\Ee$ be a semiadditive category, and define a global category $\ul{\mathrm{Mack}}_{\Ee}$ by $\ul{\mathrm{Mack}}_{\Ee}(G) = \Fun^{\times}(\Span(\Fin_G),\Ee)$, the category of $G$-Mackey functors in $\Ee$. We will show that $\ul{\mathrm{Mack}}_{\Ee}$ admits abstract Adams isomorphisms.

		Note that for an arbitrary finite groupoid $A$, viewed as a presheaf on $\Glo$, one has
        \[
        \ul{\mathrm{Mack}}_{\Ee}(A) = \Fun^{\times}(\Span(\Fin_A),\Ee).
        \]
          One can show that for any fully faithful functor $\alpha\colon A\rightarrow B$ between finite groupoids, the induction-restriction adjunction $\alpha_! \colon \Fin_A\rightleftarrows \Fin_B \noloc \alpha^*$ satisfies the assumptions of \cite{BachmannHoyois2021Norms}*{Corollary C.21}, and so we obtain a pair of functors
		\[
		\alpha_!  \colon \Span(\Fin_A)\rightleftarrows \Span(\Fin_B) \noloc \alpha^*
		\]
		such that $\alpha^*$ is both left and right adjoint to $\alpha_!$. The (co)unit of the adjunction is induced by the (co)unit of the original adjunction, and so the fact that these functors satisfy the appropriate base change can be deduced from the case of $\ul{\mathbb{F}}_{\Glo}^{\Orb}$. Applying \cite{CLL_Global}*{Remark 4.3.12}, we conclude that $\Span(\ul{\mathbb{F}}_{\Glo}^{\Orb})$ is equivariantly semiadditive, and by 2-functoriality of $\Fun^{\times}(-,\Ee)$ we deduce the same for $\ul{\mathrm{Mack}}_{\Ee}$. If $\Ee$ is furthermore presentable, then $\mathrm{Mack}_{\Ee}$ is fiberwise presentable by \cite{HTT}*{Proposition 5.5.8.10(1)} and so also equivariantly presentable.

		It remains to show that the fixed point functors $(-)^N\colon \mathrm{Mack}_{\Ee}(G) \to \mathrm{Mack}_{\Ee}(G/N)$ preserve colimits. This is clear for sifted colimits, as these are computed pointwise in $\mathrm{Mack}_{\Ee}(G)$ \cite{HTT}*{Proposition 5.5.8.10(4)} and $(-)^N$ is given by precomposition. Because finite coproducts in $\mathrm{Mack}_{\Ee}(G)$ are also products, it follows that $(-)^N$ commutes with all colimits. We conclude that $\ul{\mathrm{Mack}}_{\Ee}$ admits abstract Adams isomorphisms.
	\end{example}

	\begin{example}[Representation theory]
		\label{ex:RepTheory}
		Let $\Ee$ be a presentable semiadditive category. Then one may consider the global category $\Cc(G)\coloneqq\Fun(BG,\Ee)$, which is equivariantly presentable and equivariantly semiadditive by Proposition 4.3.3 and Proposition 4.4.9 of \cite{hopkins2013ambidexterity}, respectively. For example:
		\begin{itemize}
			\item For $\Ee = \mathrm{Vect}_k$ the category of $k$-vector spaces over a field $k$, $\Cc(G)$ is the category $\mathrm{Rep}_k(G)$ of $G$-representations over $k$;
			\item For $\Ee = \Dd(R)$ the derived category of a ring $R$, $\Cc(G)$ is the derived category $\Dd(R[G])$ of the group ring $R[G]$.
		\end{itemize}
		In this case, we do not in general have an abstract Adams isomorphism as the fixed point functors do not need to preserve colimits: as a concrete example, the functor $(-)^{C_2}\colon \Fun(BC_2,\mathrm{Vect}_{\mathbb{F}_2}) \to \mathrm{Vect}_{\mathbb{F}_2}$ does not preserve the exact sequence $\mathbb{F}_2[C_2] \to \mathbb{F}_2 \to 0$. The following lemma gives a complete characterization for when $\Cc$ has abstract Adams isomorphisms, in terms of the notion of \textit{1-semiadditivity} from \cite{hopkins2013ambidexterity}*{Definition~4.4.2}:
	\end{example}

	\begin{lemma}
		Let $\Ee$ be a semiadditive category. The global category $\Cc(G) \coloneqq \Fun(BG,\Ee)$ admits abstract Adams isomorphisms if and only if $\Ee$ is 1-semiadditive.
	\end{lemma}
	\begin{proof}
		Since the category $\Fun(BG,\Ee)$ is generated under colimits by objects induced from the trivial group, \cite{hopkins2013ambidexterity}*{Lemma~4.3.8}, it follows that every object $X \in \Cc(G)$ is $N$-free for $N \leqslant G$. In particular, $\Cc$ admits abstract Adams isomorphisms if and only if the $N$-fixed point functors $(-)^N = (-)^{hN}\colon \Cc(G) \to \Cc(G/N)$ preserve colimits. Since any morphism of groups factors as a quotient map followed by an injection, this is equivalent to the condition that $f_*$ preserves colimits for any morphism $f\colon H \to G$ of finite groups. But by \cite{hopkins2013ambidexterity}*{Proposition~4.3.9}, this is equivalent to the condition that $\Ee$ is 1-semiadditive.
	\end{proof}

	\begin{example}[Equivariant Kasparov categories]
		\label{ex:Kasparov}
		For a finite group $G$, we may consider the \textit{$G$-equivariant Kasparov category} $\KK^G = \Ind(\KK^G_{\sep})$ from \cite{BEL2023Kasparov}. We claim that these categories assemble into an equivariantly presentable and equivariantly semiadditive global category $\ul{\KK}$, which we will call the \textit{global Kasparov category}. We will further see that $\ul{\KK}$ satisfies the conditions of \Cref{thm:GlobalFormalAdamsIso}, thus providing formal Adams isomorphisms for equivariant Kasparov categories.

        We construct $\ul{\KK}$ using the following four steps:
        \begin{enumerate}[(1)]
            \item We start by considering the global category $\ul{C^*\mathrm{Alg}}^{\mathrm{nu}}_{\mathrm{sep}} \colon \Glo\catop \to \Cat$ given by the assignment $G \mapsto \Fun(BG,C^*\mathrm{Alg}^{\mathrm{nu}}_{\mathrm{sep}})$, where $C^*\mathrm{Alg}^{\mathrm{nu}}_{\mathrm{sep}}$ denotes the 1-category of separated non-unital $C^*$-algebras.
            \item For every finite group $G$, there is a class of morphisms in $\Fun(BG,C^*\mathrm{Alg}^{\mathrm{nu}}_{\mathrm{sep}})$ called the \textit{$\mathrm{kk}_0^G$-equivalences} by \emph{op.~cit.}. By \cite[Theorem~1.22]{BEL2023Kasparov}, the restriction functor
            \[
            \mathrm{Res}^G_H \colon \Fun(BG,C^*\mathrm{Alg}^{\mathrm{nu}}_{\mathrm{sep}}) \to \Fun(BH,C^*\mathrm{Alg}^{\mathrm{nu}}_{\mathrm{sep}})
            \]
            for a group homomorphism $H \to G$ sends $\mathrm{kk}_0^G$-equivalences to $\mathrm{kk}_0^H$-equivalences. In particular, $\ul{C^*\mathrm{Alg}}^{\mathrm{nu}}_{\mathrm{sep}}$ refines to a functor $\ul{C^*\mathrm{Alg}}^{\mathrm{nu}}_{\mathrm{sep}}\colon \Glo\catop \to \RelCat$, where $\RelCat$ denotes the category of \textit{relative categories}.
            \item Composing $\ul{C^*\mathrm{Alg}}^{\mathrm{nu}}_{\mathrm{sep}}$ with the Dwyer-Kan localization functor $L\colon \RelCat \to \Cat$ gives a global category $\ul{\KK}_{\sep}\colon \Glo\catop \to \Cat$. By definition its value at $G$ is the category $\KK^G_{\sep}$ of \cite[Definition~1.2]{BEL2023Kasparov}.
            \item Finally, we define the global category $\ul{\KK}$ as the composite
            \[
                \ul{\KK} \colon \Glo\catop \xrightarrow{\ul{\KK}_{\sep}} \Cat \xrightarrow{\Ind} \PrL.
            \]
            It is given on objects by $G \mapsto \KK^G$ and on morphisms by sending a group homomorphism $H \to G$ to the restriction functor $\mathrm{Res}^G_H\colon \KK^G \to \KK^H$ constructed in \cite[Theorem~1.22]{BEL2023Kasparov}.
        \end{enumerate}

	We will now show that $\ul{\KK}$ is both equivariantly presentable and equivariantly semiadditive. Fiberwise semiadditivity is clear, as $\ul{\KK}$ is even fiberwise stable by Theorem 1.3 of \emph{op.~cit}. Given an inclusion $H \leqslant G$ of finite groups, the adjunction $\mathrm{Res}^G_H \dashv \Ind^G_H$ at the $C^*$-algebra level\footnote{We follow \cite{BEL2023Kasparov} in denoting the right adjoint of restriction by $\Ind^G_H$ in this case. We emphasize that, at the level of $C^*$-algebras, the functor $\Ind^G_H$ is only a \textit{right} adjoint to restriction, not a left adjoint.} descends by Theorem 1.22 of \emph{op.~cit.} to the level of equivariant Kasparov categories. Since the Beck-Chevalley conditions are satisfied at the level of $C^*$-algebras, they descend to the level of equivariant Kasparov categories, showing that the global Kasparov category is equivariantly complete. Furthermore, Theorem 1.23(1) of \emph{op.~cit.} shows that the functor $\Ind^G_H\colon \KK^H \to \KK^G$ is also a \textit{left} adjoint to $\mathrm{Res}^G_H$, as exhibited by an explicit unit transformation $\id \to \mathrm{Res}^G_H\Ind^G_H$ constructed at the level of $C^*$-algebras. One can deduce from \cite[Lemma~4.3.11]{CLL_Global} that this map agrees with the dual adjoint norm map $\Nmadjdual\colon \id \to \mathrm{Res}^G_H \Ind^G_H$ constructed in \cite[Construction~4.3.8]{CLL_Global}, and using fiberwise semiadditivity of $\ul{\KK}$ it then follows from \cite[Lemma~4.5.4]{CLL_Global} that $\ul{\KK}$ is even equivariantly semiadditive. Finally, equivariant presentability of $\ul{\KK}$ is a consequence of fiberwise presentability and equivariant semiadditivity.

    Since the inflation functor $\KK^{G/N} \to \KK^G$, given by restriction along the quotient map $G \to G/N$, preserves compact objects by construction of $\ul{\KK}$, its right adjoint $(-)^N\colon \KK^G \to \KK^{G/N}$ preserves colimits; it thus follows from \Cref{thm:GlobalFormalAdamsIso} that $\ul{\KK}$ admits abstract Adams isomorphisms.
	\end{example}

	\begin{remark}
   	In the previous example, the abstract Adams isomorphism is not optimal: it is proven in \cite[Theorem~1.23]{BEL2023Kasparov} that the left and right adjoint to the inflation functor $\KK \to \KK^G$ agree on \textit{all} objects in $\KK^G$, not just the $G$-free ones, and this is expected to hold more generally for the inflation functor $\KK^{G/N} \to \KK^G$ for $N\trianglelefteqslant G$.
	\end{remark}

        \appendix

        \section{Cocontinuity for the universal examples} \label{sec:cocont}

        Consider an orbital category $T$ and an atomic orbital subcategory $P \subseteq T$. In \cite{CLL_Clefts}, the authors constructed the universal example of a $P$-semiadditive (resp.\ $P$-stable) $P$-presentable $T$-category, denoted by $\ul\CMon^P_{P\triangleright T}$ (resp.\ $\ul\Sp^P_{P\triangleright T}$). Similarly, \cite{CLL_Global} constructs the universal examples $\ul\CMon^P_T$ and $\ul\Sp^P_{T}$ of $T$-presentable $P$-semiadditive resp.\ $P$-stable $T$-categories.

		The goal of this appendix is to show that the cocontinuity condition of \Cref{thm:adams-gen} is always satisfied for these universal $T$-categories:

	\begin{theorem}\label{thm:universal-cocontinuous}
		Let $f\colon A\to B$ be any map in $T$. Then each of the functors
		\begin{align*}
			f_*\colon\ul\CMon^P_T(A)&\to\ul\CMon^P_T(B)\\
			f_*\colon\ul\CMon^P_{P\triangleright T}(A)&\to\ul\CMon^P_{P\triangleright T}(B)\\
			f_*\colon\ul\Sp^P_T(A)&\to\ul\Sp^P_T(B)\\
			f_*\colon\ul\Sp^P_{P\triangleright T}(A)&\to\ul\Sp^P_{P\triangleright T}(B)
		\end{align*}
		is cocontinuous.
	\end{theorem}

	As recalled in Section~\ref{sec:recollections}, for $\Orb\subseteq\Glo$ these universal examples can be explicitly described via global or equivariant $\Gamma$-spaces and global or equivariant spectra, respectively, and this provides the cocontinuity used in Examples~\ref{ex:GlobalEquivHtpyTheory} and~\ref{ex:gamma-Adams}.

	For the proof of the theorem, we recall that \cite{CLL_Global}*{Definition~4.8.1} constructs $\ul\CMon^P_T$ as a certain subcategory of the internal hom $\ul\Fun_T(\ul{\mathbb F}^P_{T,*},\ul\Spc_T)$, where $\ul{\mathbb F}^P_{T,*}$ denotes the $T$-category of pointed objects in $\ul{\mathbb F}^P_T$. We therefore begin with a statement about general $T$-categories of copresheaves:

	\begin{proposition}\label{prop:T-spc-sifted}
		Let $I$ be any $T$-category and let $f\colon A\to B$ lie in $\ul{\mathbb F}_T$. Then $f_*\colon\ul\Fun_T(I,\ul\Spc_T)(A)\to\ul\Fun_T(I,\ul\Spc_T)(B)$ preserves sifted colimits.
		\begin{proof}
			We will prove the proposition by treating successively more general cases.

			\smallskip
			\textit{Step 1.} We first treat the special case that $I$ is terminal and the presheaves $A$ and $B$ are representable, so that $f$ is a map in $T$. Identifying $\ul\Spc_T(A)=\PSh(T)_{/A}\simeq\PSh(T_{/A})$ with $\Fun^{\times}(\ul{\mathbb F}_{T}(A)^\op,\Spc)$ and $\ul\Spc_T(B)\simeq\Fun^\times(\ul{\mathbb F}_T(B)^\op,\Spc)$, the structure map $f^*\colon\ul\Spc_T(B)\to\ul\Spc_T(A)$ is given by restriction along the opposite of $f_!\colon\ul{\mathbb F}_T(A)\to\ul{\mathbb F}_T(B), (X\to A)\mapsto (X\to A\to B)$. However, $f_!$ has a right adjoint $f^*$ (given by pulling back along $f$) by orbitality of $T$, and this preserves finite coproducts as the corresponding map of presheaf categories is even a left adjoint. Thus, under the chosen identifications $f_*\colon\ul\Spc_T(A)\to\ul\Spc_T(B)$ is simply given by restricting along the opposite of $f^*\colon \ul{\mathbb F}_T(B)\to\ul{\mathbb F}_T(A)$.

			We now observe that $\Fun^{\times}(\ul{\mathbb F}_{T}(A)^\op,\Spc)\subseteq\PSh(\ul{\mathbb F}_T(A))$ is closed under sifted colimits as sifted colimits in spaces commute with finite products, and similarly for $\Fun^{\times}(\ul{\mathbb F}_{T}(B)^\op,\Spc)$. Thus, sifted colimits on both sides can be computed pointwise and the claim follows immediately.

			\smallskip
			\textit{Step 2.} Next, we consider the case that $I$ is terminal and $B$ is representable, so that $f$ decomposes as a finite coproduct $(f_i)\colon\coprod_{i=1}^nA_i\to B$ with $A_i$ representable. Then $f_*$ is simply given as the product of the individual $f_{i*}$, so the claim follows from the previous step, using again that sifted colimits commute with products.

			\smallskip
			\textit{Step 3.} Now we can treat the case that $I$ is terminal and $f\colon A\to B$ is a general map in $\ul{\mathbb F}_T$. Writing $B$ as a colimit of representables, we see that the restrictions $t^*\colon\ul\Spc_T(B)\to\ul\Spc_T(C)$ for all maps $t\colon C\to B$ with $C\in T$ are jointly conservative. As they are moreover left adjoints, it will be enough to show that $t^*f_*$ preserves sifted colimits for each such $t$. For this we consider the pullback diagram
			\begin{equation*}
				\begin{tikzcd}
					D\arrow[d, "g"']\arrow[r, "u"]\arrow[dr,phantom,"\lrcorner"{very near start}] & A\arrow[d, "f"]\\
					C\arrow[r, "t"'] & B\rlap.
				\end{tikzcd}
			\end{equation*}
			By the Beck-Chevalley condition for the $T$-complete category $\ul\Spc_T$, we get $t^*f_*\simeq g_*u^*$. On the other hand, $g$ is a map in $\ul{\mathbb F}_T$, so the previous step shows that $g_*$ preserves sifted colimits. As $u^*$ is a left adjoint, the claim follows.

			\smallskip
			\textit{Step 4.} Now assume $I$ is a $T$-presheaf. Then $f^*\colon\ul\Fun_T(I,\Cc)(B)\to\ul\Fun_T(I,\Cc)(A)$ agrees up to the equivalences from \cite{CLL_Global}*{Corollary~2.2.9} with the restriction $(f\times I)^*\colon\Cc(B\times I)\to\Cc(A\times I)$. As $f\times I$ is a pullback of $f$, it again lies in $\ul{\mathbb F}_T$, so $(f\times I)_*$ preserves sifted colimits by the previous step.

			\smallskip
			\textit{Step 5.} Now we treat the general case. For each $T$-functor $t\colon J\to I$, the functor $t^*\colon\ul\Fun_T(I,\Cc)\to\ul\Fun_T(J,\Cc)$ preserves both $T$-limits and colimits by \cite{martiniwolf2021limits}*{Proposition~4.3.1}. In particular, the previous step shows that \[t^*f_*\simeq f_*t^*\colon\ul\Fun_T(I,\Cc)(A)\to\ul\Fun_T(J,\Cc)(B)\] preserves sifted colimits whenever $J$ is a $T$-presheaf. As the functors $t^*$ for $t$ running through all maps from $T$-presheaves to $I$ are jointly conservative by \cite{martini2021yoneda}*{Corollary~4.7.17}, the claim follows.
		\end{proof}
	\end{proposition}

	In order to apply this to $\ul\CMon^P_T$ we will use:

	\begin{lemma}\label{lemma:cmon-sifted-closed}
		Let $A\in\PSh(T)$. Then the full subcategory \[\ul\CMon^P_T(A)\subseteq\ul\Fun_T(\ul{\mathbb F}^P_{T,*},\ul\Spc_T)(A)\] is closed under sifted colimits.
		\begin{proof}
			Decomposing $A$ into representables and using that restriction functors are cocontinuous, we may assume that $A$ itself is representable.

			We write $\smallint\ul{\mathbb F}^P_{T,*}\times\ul A$ for the cocartesian unstraightening of the functor $T^\op\to\Cat,\allowbreak B\mapsto\ul{\mathbb F}^P_{T,*}(B)\times\Hom(B, A)$, and we will as usual denote objects in this unstraightening by triples $(B,X_+,f)$ with $B\in T^\op$, $X_+\in\ul{\mathbb F}^P_{T,*}(B)$, and $f\colon B\to A$. Then \cite{CLL_Global}*{Remark~2.2.14} provides an equivalence
			\begin{equation*}
				\ul\Fun_T(\ul{\mathbb F}^P_{T,*},\ul\Spc_T)(A)\simeq\Fun(\smallint \ul{\mathbb F}^P_{T,*}\times\ul A,\Spc)
			\end{equation*}
			and Remark 4.9.9 of \emph{op.\ cit.}\ shows that $\ul\CMon^P_T(A)$ corresponds under this equivalence to the full subcategory of all $F\colon\smallint \ul{\mathbb F}^P_T\times\ul A \to\Spc$ satisfying the following conditions:
			\begin{enumerate}
				\item For every $f\colon B\to A$ in $T$ the restriction of $F$ to the non-full subcategory $\ul{\mathbb F}^P_T(B)\times\{f\}$ is semiadditive in the usual, non-parametrized sense.
				\item For every $p\colon B\to A$ in $P$ and $f\colon A\to A'$ in $T$ a certain natural \emph{Segal map} $F(A,B_+,f)\to F(B,B_+,fp)$ is an equivalence.
			\end{enumerate}
			Obviously, the first property is stable under sifted colimits, while the second property is even stable under arbitrary colimits.
		\end{proof}
	\end{lemma}

	\begin{remark}
		In fact, both the above lemma as well as its proof do not need any orbitality assumption on $T$.
	\end{remark}

	\begin{proof}[Proof of Theorem~\ref{thm:universal-cocontinuous}]
		For $\ul\CMon^P_T$, we observe that as $f_*$ is a right adjoint, it in particular preserves finite products. Since both source and target are semiadditive, we conclude that $f_*$ also preserves finite \emph{co}products, so it only remains to show that it also preserves sifted colimits. The latter is however immediate by combining Proposition~\ref{prop:T-spc-sifted} with Lemma~\ref{lemma:cmon-sifted-closed}, finishing the argument for $\ul\CMon^P_T$.

		For $\ul\CMon^P_{P\triangleright T}$, we recall from \cite{CLL_Clefts}*{Theorems~6.18 and 6.19} that there exists a fully faithful $T$-functor $\iota_!\colon\ul\CMon^P_{P\triangleright T}\to\ul\CMon^P_T$, such that $\iota_!\colon\ul\CMon^P_{P\triangleright T}(X)\to\ul\CMon^P_T(X)$ sits in a sequence of adjoint functors $\iota_!\dashv\iota^*\dashv\iota_*$ for every $X\in T$. The total mate of the structure equivalence $\iota_!f^*\simeq f^*\iota_!$ then provides an equivalence $f_*\iota^*\simeq \iota^*f_*$. Combining this with full faithfulness of $\iota_!$, we then conclude that $f_*\colon\ul\CMon^P_{P\triangleright T}(A)\to\ul\CMon^P_{P\triangleright T}(B)$ agrees with the composite
		\begin{equation*}
			\ul\CMon^P_{P\triangleright T}(A)\xrightarrow{\;\iota_!\;}\ul\CMon^P_T(A)\xrightarrow{\;f_*\;}\ul\CMon^P_T(B)\xrightarrow{\;\iota^*\;}\ul\CMon^P_T(B),
		\end{equation*}
		each of which is cocontinuous.

		Next, we recall \cite{CLL_Clefts}*{Definition~8.9} that $\ul\Sp^P_T$ is obtained from $\ul\CMon^P_T$ by pointwise taking the tensor product in $\PrL$ with the category $\mySp$ of spectra. By the above, $f^*\colon\ul\CMon^P_T(B)\rightleftarrows\ul\CMon^P_T(A)\noloc f_*$ is an internal adjunction in $\Pr^\text{L}$. By $2$-functoriality of $\mySp\otimes\blank$, we can therefore identify $f_*\colon\ul\Sp^P_T(B)\to\ul\Sp^P_T(A)$ with the cocontinuous functor $\mySp\otimes f_*\colon {\mySp}\otimes{\ul\CMon^P_T}(A)\to{\mySp}\otimes{\ul\CMon^P_T}(B)$.

		The argument for $\ul\Sp^P_{P\triangleright T}$ is analogous.
	\end{proof}

        \section{Comparison with Sanders' approach}
        \label{sec:Comparison_Sanders}
        In \cite{Sanders2019Adams}*{Theorem~2.9}, Sanders provides an alternative treatment of formal Adams isomorphisms. In this appendix, we discuss the relationship between his approach and ours.

        Sanders' theorem is formulated in the context of \textit{geometric functors} $f^*\colon \Dd \to \Cc$ between \textit{rigidly-compactly generated tensor-triangulated categories}. In the language of higher category theory, this corresponds to the condition that $\Dd$ and $\Cc$ are compactly generated stable symmetric monoidal $\infty$-categories whose compact objects are precisely the dualizable objects, and $f^*$ is a symmetric monoidal colimit-preserving functor. This implies that $f^*$ admits a right adjoint $f_*$ satisfying the projection formula, which in turn admits a further right adjoint $f^!$.

        Consider now a full subcategory $\Bb \subseteq \Cc$, and suppose that the inclusion $i_!\colon \Bb \hookrightarrow \Cc$ fits in an adjoint triple $i_! \dashv i^* \dashv i_*$, so that $\Bb$ is in particular closed under colimits in $\Cc$. In total we have
        \[
        \begin{tikzcd}
        	\Bb \ar[rr, hookrightarrow, shift left = 5.25, "i_!"] \ar[rr,leftarrow, shift left = 1.75, "i^*"{description}] \ar[rr,hookrightarrow, shift right = 1.75, "i_*"{description}] \ar[rr,leftarrow, dashed, shift right = 5.25, "\times"{description}]
        	&& \Cc \ar[rr, dashed, shift left = 5.25, "\times"{description}] \ar[rr,leftarrow, shift left = 1.75, "f^*"{description}] \ar[rr,shift right = 1.75, "f_*"{description}] \ar[rr,leftarrow, shift right = 5.25, "f^!"']
        	&& \Dd,
        \end{tikzcd}
        \]
        where the dotted maps indicate that the adjoints do not necessarily exist. Sanders proves the following theorem:

        \begin{theorem}[{\cite{Sanders2019Adams}*{Theorem~2.9}}]
        \label{thm:Sanders}
            Consider the previous situation, and assume further that $\Bb$ is generated under colimits by the thick tensor ideal
            \[
            \Ii \coloneqq \{b \in \Bb \cap \Cc^c \mid \text{$f_*(c \otimes b) \in \Dd^c$ for all $c \in \Cc^c$}\},
            \]
            where $\Cc^c \subseteq \Cc$ and $\Dd^c \subseteq \Dd$ denote the subcategories of compact (equivalently: dualizable) objects. Define $\omega_f \coloneqq f^!\unit_{\Dd}$. Then:
        \begin{enumerate}[(a)]
            \item The composite $i^*f^!\colon \Dd \to \Bb$ has a right adjoint;
            \item There is a canonical equivalence $i^*(\psi)\colon i^*(\omega_f \otimes f^*(-)) \iso i^* f^!$;
            \item There is a canonical equivalence $i^*(\psi^\mathrm{ad})\colon i^* f^* \iso i^*(\iHom(\omega_f, f^!(-)))$;
            \item The composite $i^*f^*\colon \Dd \to \Bb$ admits a left adjoint, which we will denote by
            \[
            (fi)_! \colon \Bb \to \Dd,
            \]
            and there is a canonical equivalence
            \[
            \varpi \colon (fi)_! \iso f_*(i_!(-) \otimes \omega_f),
            \]
            which we call the \emph{Wirthm\"uller isomorphism} associated to $f^*$ and $\Bb$.
        \end{enumerate}
        \end{theorem}

        In order to be able to compare the Wirthm\"uller isomorphism $(fi)_! \simeq f_*(i_!(-) \otimes \omega_f)$ of Sanders to our abstract Adams isomorphism, let us briefly summarize the proof of \Cref{thm:Sanders}:
        \begin{enumerate}
        	\item Part (a) follows from the fact that left adjoint $f_* i_!\colon \Bb \to \Cc$ of $i^* f^!$ sends the generating set of compact objects $\Ii \subseteq \Bb$ to compact objects in $\Dd$, by definition of $\Ii$.
        	\item Part (b) is equivalent to the statement that the lax $\Dd$-linear functor $i^* f^!$ is in fact $\Dd$-linear; since it preserves colimits by (a) and its left adjoint $f_* i_!$ is $\Dd$-linear, this is an instance of a general fact: for $\Dd$ rigidly compactly generated, the right adjoint $G\colon \Nn \to \Mm$ of a $\Dd$-linear left adjoint $F\colon \Mm \to \Nn$ between $\Dd$-linear categories is again $\Dd$-linear as soon as $G$ preserves colimits.
        	\item  To construct the left adjoint in part (d), it suffices to do so objectwise on the generators $b \in \Ii \subseteq \Bb$. There one can show that we have $b^{\vee} \in \Ii$, and a computation shows that the desired object $(f i)_!(b)$ is given by $f_*(b^{\vee})^{\vee}$.
        	\item Finally, checking that the canonical comparison maps in (c) and (d) are equivalences is quite intricate, see \cite{Sanders2019Adams}*{Proposition~2.31}. However, when there is an equivalence $i^*\omega_f \simeq \unit_{\Bb}$ (as will be the case in the context we are interested in) the equivalences in (c) and (d) are direct consequences of the equivalence of (b).
        \end{enumerate}

        We now have the following comparison result:
        \begin{proposition}
        \label{prop:Comparison_Sanders}
            Let $\Cc$ be a symmetric monoidal $T$-category, that is, a functor $\Cc\colon T\catop \to \CMon(\Cat)$. Assume that $\Cc(C)$ is rigidly compactly generated for all $C \in T$. Let $P \subseteq T$ be an atomic orbital subcategory and assume that $\Cc$ is $P$-presentable and $P$-semiadditive, and that the tensor product $-\otimes- \colon \Cc\times \Cc\rightarrow \Cc$ preserves $P$-colimits in each variable, in the sense of \cite{martiniwolf2022presentable}*{Definition 8.1.1}. Let $f\colon A \to B$ be a morphism in $T$ whose diagonal is in $\ulfinPsets$. Then:
            \begin{enumerate}
            	\item The functor $f_*\colon \Cc(A)\rightarrow \Cc(B)$ preserves colimits, and so \Cref{thm:adams-gen} provides an abstract Adams isomorphism $\Nm_f^{\fr}\colon f^{\fr}_!\iso f_*$ on $f$-free objects.
                \item The functor $f^*\colon \Cc(B) \to \Cc(A)$, equipped with the subcategory
                \[
                \Bb \coloneqq \Cc(A)_{\free{f}} \subseteq \Cc(A)
                \]
                of $f$-free objects, satisfies all the assumptions of \Cref{thm:Sanders}, thus giving a Wirthm\"uller isomorphism
                \[
                    \varpi\colon (f i)_! \iso f_*(i_!(-) \otimes \omega_f).
                \]
                \item There exists a preferred equivalence $i^*\omega_f \simeq \unit_{\Cc(\EF_f)}$ such that the composite
                \[
                    (f i)_! \xrightarrow{\varpi} f_*(i_!(-) \otimes \omega_f) \simeq f_*i_!(- \otimes i^*\omega_f) \simeq f_*i_!(-)
                \]
                agrees with the norm map $\Nm_f\colon (f i)_! = f^{\fr}_! \iso f_*i_!(-)$.
            \end{enumerate}
        \end{proposition}

        In part (3), we have identified $\Cc(A)_{\free{f}}$ with $\Cc(\EF_f)$ in light of \Cref{lem:BorelObjectsOverBF}.

        \begin{proof}
            Part (1) follows immediately from the fact that $f^*$ is by assumption a geometric functor between rigidly compactly generated categories. Next we show part (2). For this we further note that the inclusion $i_!\colon \Bb = \Cc(A)_{\free{f}} \hookrightarrow \Cc(A)$ admits a right adjoint $i^*$ which admits a further right adjoint $i_*$ by \Cref{lem:BorelObjectsOverBF}. It remains to check that $\Bb$ is generated under colimits by the subcategory
            \[
            \Ii \coloneqq \{b \in \Bb \cap \Cc(A)^c \mid \text{$f_*(c \otimes b) \in \Cc(B)^c$ for all $c \in \Cc(A)^c$}\}.
            \]
            By \Cref{cor:BorelObjectsGenerated}, $\Bb$ is generated under colimits by the objects of the form $i_!X$ for $i\colon C \to A$ in $\Ff \subseteq P_{/A}$ and $X \in \Cc(C)$. Since $\Cc(C)$ is compactly generated, we may in fact restrict to compact objects $X \in \Cc(C)$. It thus remains to show that $i_!X$ is contained in $\Ii$ for every compact object $X \in \Cc(C)$. In other words, we have to show that for every $c \in \Cc(A)^c$, the object $f_*(c \otimes i_!X)$ is again compact. But this is a consequence of the equivalences
            \[
                f_*(c \otimes i_!X) \simeq f_*i_!(i^*c \otimes X) \overset{\Nm_i}{\simeq} f_*i_*(i^*c \otimes X) \overset{\Nm_{fi}}{\simeq} (fi)_!(i^*c \otimes X),
            \]
            using the fact that the functors $i^*$ and $(fi)_!$ preserves compact objects (as their right adjoints preserve colimits) and the fact that compact objects in $\Cc(C)$ are closed under tensor product (as they agree with the dualizable objects). This finishes the proof of part (2).

            For part (3), consider the abstract Adams isomorphism $\Nm_f\colon f_!^{\fr} \iso f_*i_!(-)$ from \Cref{thm:adams-gen}. Passing to right adjoints gives an equivalence $(\Nm_f)^{\mathrm{ad}}\colon i^*f^* \iso i^*f^!$, and evaluating at $\unit_{\Cc(B)}$ provides an equivalence
            \[
                (\Nm_f)^{\mathrm{ad}} \unit_{\Cc(B)} \colon \unit_{\Cc(\EF_f)} \simeq i^*f^*\unit_{\Cc(B)} \iso i^*f^!\unit_{\Cc(B)} = i^*\omega_f.
            \]
            It remains to show that under this equivalence, the norm map $\Nm_f$ agrees with the composite
            \[
                (f i)_!(-) \xrightarrow{\varpi} f_*(i_!(-) \otimes \omega_f) \simeq f_*i_!( - \otimes i^*\omega_f),
            \]
            where the last equivalence is the projection formula for $i_!$. Equivalently, by passing to right adjoints, it will suffice to show that the equivalence $i^*f^* \iso i^*f^!$ agrees under this equivalence with the map
            \[
                i^*f^*(-) \xrightarrow{i^*(\psi^{\mathrm{ad}})} i^*(\iHom(\omega_f, f^!(-)) \simeq \iHom(i^*\omega_f, i^*f^!(-)),
            \]
            where the last equivalence is the closed monoidal structure on $i^*$, obtained by adjunction from the projection formula for $i_!$. By the definition of the maps $\psi$ and $\psi^{\mathrm{ad}}$ in \cite[Remark~2.15]{Sanders2019Adams}, the latter composite is adjoint to the map
            \[
                i^*\omega_f \otimes i^*f^*(-) \simeq i^*(\omega_f \otimes f^*(-)) \xrightarrow{i(\psi)} i^*f^!(-),
            \]
            and hence it will suffice to show that the following diagram commutes:
            \[
            \begin{tikzcd}
                \unit_{\Cc(\EF_f)} \otimes i^*f^*(-) \ar{rr}{\simeq} \dar[swap]{(\Nm_f)^{\mathrm{ad}} \unit_{\Cc(B)} \otimes \id} && i^*f^*(-) \dar{(\Nm_f)^{\mathrm{ad}}} \\
                i^*\omega_f \otimes i^*f^*(-) \rar{\simeq} & i^*(\omega_f \otimes f^*(-)) \rar{i(\psi)} & i^*f^!(-).
            \end{tikzcd}
            \]
            Observe that this is a special case of the statement that the equivalence \[(\Nm_f)^{\mathrm{ad}}\colon i^*f^* \iso i^*f^!\] is $\Cc(B)$-linear, in the sense that for all objects $X, Y \in \Cc(B)$ the following diagram commutes:
            \[
            \begin{tikzcd}
                i^*f^*(X) \otimes i^*f^*(Y) \ar{rr}{\simeq} \dar[swap]{(\Nm_f)^{\mathrm{ad}}X \otimes \id} && i^*f^*(X \otimes Y) \dar{(\Nm_f)^{\mathrm{ad}}} \\
                i^*f^!(X) \otimes i^*f^*(Y) \rar{\simeq} & i^*(f^!(X) \otimes f^*(Y)) \rar{\simeq} & i^*f^!(X \otimes Y).
            \end{tikzcd}
            \]
            By passing to left adjoints, this is equivalent to the norm map $\Nm_f^{\fr}\colon f_!^{\fr} \iso f_*i_!$ being $\Cc(B)$-linear, which in turn is equivalent to the map $\smash{\Nmadjdual}^{\fr}_f\colon i^* \to i^*f^*f_*$ being $\Cc(B)$-linear, i.e.~for all $X \in \Cc(A)$ and $Y \in \Cc(B)$ the following diagram commutes:
            \[
            \begin{tikzcd}
                i^*(X) \otimes i^*f^*(Y) \ar{rr}{\simeq} \dar[swap]{\smash{\Nmadjdual}^{\fr}_f \otimes \id} && i^*(X \otimes f^*Y) \dar{\smash{\Nmadjdual}^{\fr}_f} \\
                i^*f^*f_*(X) \otimes i^*f^*(Y) \rar{\simeq} & i^*f^*(f_*(X) \otimes Y) \rar{\simeq} & i^*f^*f_*(X \otimes f^*Y).
            \end{tikzcd}
            \]
            This statement is a somewhat tedious exercise in calculus of mates: one needs to show that all the Beck-Chevalley transformations used in the definition $\smash{\Nmadjdual}^{\fr}_f$ are compatible with the projection formula equivalences. We will only sketch the proof. Looking at the zig-zag \eqref{eq:Nm-zig-zag} on page \pageref{eq:Nm-zig-zag}, one observes that that the maps $\epsilon$ and $\BC_*$ are compatible with the projection formula equivalences, and it thus remains to show the map $\Nm_{\Delta}$ is also compatible. More generally, we will argue that for every morphism $p\colon A \to B$ in $\ulfinPsets$ the norm equivalence $\Nm_{p}\colon p_! \iso p_*$ is compatible with the projection formula equivalences. By adjunction, it again suffices to show this for the map $\Nmadjdual_p \colon \id \to p^*p_*$ from \cite[Construction~4.3.8]{CLL_Global}, which is defined using the analogous zig-zag \eqref{eq:Nm-zig-zag}. We thus obtain a further reduction to the statement for $\Nm_{\Delta_p}\colon (\Delta_p)_! \iso (\Delta_p)_*$. Since $\Delta_p$ is a disjoint summand inclusion in $\PSh(T)$, the statement here follows from the uniqueness statement from \cite[Lemma~4.1.4]{CLL_Global}.
        \end{proof}

		\begin{remark}
		In the language of \cite{Sanders2019Adams}, we have in particular shown that the category $\Cc(A)_{\free{f}}$ of $f$-free objects is contained in the \emph{compactness locus} of the functor $f^*$. In general the compactness locus may be larger, for example in the case of the inflation functor $\mySp_{G/N}\rightarrow \mySp_{G}$, see Theorem 6.2 of \textit{op.~cit.}
		\end{remark}

        \begin{remark}
            In \cite[Theorem~1.7]{BDS2016Grothendieck}, Balmer, Dell'Ambrogio, and Sanders prove a formal Wirthm\"uller isomorphism in the context of rigidly-compactly generated tensor-triangulated categories: given a geometric functor $f^*\colon \Dd \to \Cc$ whose right adjoint $f_*\colon \Dd \to \Cc$ preserves compact objects, the functor $f^*$ admits a left adjoint $f_!$ and there is an \textit{ur-Wirthm\"uller isomorphism}
            \[
                \varpi\colon f_!(-) \iso f_*(\omega_f \otimes -).
            \]
            Since the theorem by Sanders mentioned above reduces to this theorem by taking $\Bb = \Cc$, it follows that \Cref{prop:Comparison_Sanders} provides an equivalence between the ur-Wirthm\"uller isomorphism from \emph{op.~cit.} and our norm equivalences:
        \end{remark}

        \begin{corollary}
            In the situation of \Cref{prop:Comparison_Sanders}, assume the map $f\colon A \to B$  is a morphism in $P$. Then we have $\Bb = \Cc(A)$, and there exists a preferred equivalence $\omega_f \simeq \unit_{\Cc(A)}$ such that the composite
            \[
                f_!(-) \xrightarrow{\varpi} f_*(\omega_f \otimes -) \simeq f_*
            \]
            agrees with the norm map $\Nm_f\colon f_! \to f_*$.
        \end{corollary}
        \begin{proof}
            This is immediate from \Cref{prop:Comparison_Sanders}, using the observation that every object in $\Cc(A)$ is $f$-free (Example~\ref{ex:P-free}).
        \end{proof}

		\begin{remark}\label{rk:global-not-rigid}
			One can show that the usual smash product of spectra makes $\ul\mySp^\text{gl}$ into a symmetric monoidal global category such that the tensor product even preserves $\Glo$-colimits in each variable. However, in $\mySp^\text{gl}_G$ the dualizable objects do \emph{not} coincide with the compacts, so that the global Adams isomorphism is not an instance of Sanders' framework:

			If $p\colon G\to e$ denotes the unique homomorphism, then $p^*\Sigma^\infty_+ B_\text{gl}C_2\simeq p^*(\mathbb S/C_2)$ is compact (Example~\ref{ex:GlobalEquivHtpyTheory}), but its underlying non-equivariant spectrum is $\Sigma^\infty_+BC_2$, which is not compact and hence not dualizable. As the forgetful functor from $G$-global spectra to spectra is symmetric monoidal, it preserves dualizable objects, and accordingly the $G$-global spectrum $p^*\Sigma^\infty_+ B_\text{gl}C_2$ is not dualizable either.
		\end{remark}

	\section{The calculus of mates}\label{sec:BC}
	In this final appendix we recall some general identities involving Beck-Chevalley transformations (or \emph{mates}).

	Throughout, we work in a fixed ambient $(2,2)$-category $\textsf{U}$, which we for simplicity assume to be strict; the example we will apply these results to in the main text is the strict $(2,2)$-category of ($\infty$-)categories, functors, and homotopy classes of natural transformations. Recall once more that for a square
	\begin{equation}\label{diag:to-mate}
		\begin{tikzcd}
			A\arrow[r, "u^*"]\arrow[dr, Rightarrow, "\sigma"{description}] & B\\
			C\arrow[u, "f^*"]\arrow[r, "v^*"'] & D\arrow[u, "g^*"']
		\end{tikzcd}
	\end{equation}
	in which the two vertical maps admit left adjoints $f_!,g_!$, the Beck-Chevalley transformation $\BC_! = \BC_!(\sigma)\colon g_!u^*\to v^*f_!$ is defined as the composite
	\begin{equation*}
		g_!u^*\xrightarrow{\eta}g_!u^*f^*f_!\xrightarrow{\sigma}g_!g^*v^*f_!\xrightarrow{\epsilon} v^*f_!.
	\end{equation*}
	It will be convenient for us to write this composite in a more diagrammatic fashion as the \emph{pasting}
	\begin{equation*}
		\begin{tikzcd}[column sep=3em, row sep=2.5em]
			& A\arrow[r, "u^*"]\arrow[dr, Rightarrow, "\sigma"{description}] & B\arrow[r, "g_!"] & D\rlap.\\
			A\arrow[ur, bend left=15pt,"\id", ""{name=Y,marking}]\arrow[from=Y,to=r,Rightarrow,"\eta"{description}] \arrow[r, "f_!"'] & C\arrow[u, "f^*"{description}]\arrow[r, "v^*"'] & D\arrow[u, "g^*"{description}]\arrow[ur, "\id"', bend right=15pt, ""{name=X,marking,xshift=-4pt,yshift=4pt}]
				\arrow[from=u,to=X,Rightarrow,"\epsilon"{description}]
		\end{tikzcd}
	\end{equation*}
	We point out that any such pasting diagram indeed has an unambiguous composite (i.e.~the resulting natural transformation is independent of the order in which we paste the individual $2$-cells); see \cite{power-pasting} for a precise statement and proof.

	\begin{remark}
		If $u^*$ and $v^*$ admit \emph{right} adjoints $u_*,v_*$, there is a dually defined Beck-Chevalley transformation $\BC_*\colon f^*v_*\to u_*g^*$. Below we will restrict to the case of the map $\BC_!$, the other case being formally dual.
	\end{remark}

	\begin{lemma}\label{lemma:BC-unit-counit}
		Consider a square $(\ref{diag:to-mate})$ such that $f^*$ and $g^*$ admit left adjoints $f_!$ and $g_!$, respectively. Then the following diagrams of $2$-cells commute:
		\begin{equation}\label{diag:cancel-unit-counit}
			\begin{tikzcd}
				g_!u^*f^*\arrow[d, "\BC_!"']\arrow[r, "\sigma"] & g_!g^*v^*\arrow[d,"\epsilon"]\\
				v^*f_!f^*\arrow[r, "\epsilon"'] & v^*
			\end{tikzcd}
			\qquad
			\begin{tikzcd}
				u^*\arrow[r,"\eta"]\arrow[d,"\eta"'] & g^*g_!u^*\arrow[d, "\BC_!"]\\
				u^*f^*f_!\arrow[r,"\sigma"'] & g^*v^*f_!\rlap.
			\end{tikzcd}
		\end{equation}
		\begin{proof}
			We will prove the first statement (which is the only one used in the main text); the proof of the second statement is similar.

			Plugging in the definition of $\BC_!$, the lower left composite is given by the pasting
			\begin{equation}\label{diag:big-pasting}
				\begin{tikzcd}[column sep=3em, row sep=2.5em, execute at end picture={
						\fill[color=black, opacity=.075] ($(L.south west)-(4pt,0)$) rectangle ($(U.north east)-(2pt,0)$);}]
					& |[alias=U]| A\arrow[r, "u^*"]\arrow[dr, Rightarrow, "\sigma"{description}] & B\arrow[r, "g_!"] & D\rlap.\\
					A\arrow[ur, bend left=15pt,"\id", ""{name=Y,marking}]\arrow[from=Y,to=r,Rightarrow,"\eta"{description}] \arrow[r, "f_!"{description}] & C\arrow[u, "f^*"{description}]\arrow[r, "v^*"'] & D\arrow[u, "g^*"{description}]\arrow[ur, "\id"', bend right=15pt, ""{name=X,marking,xshift=-2pt,yshift=4pt}]
						\arrow[from=u,to=X,Rightarrow,"\epsilon"{description}]\\
					|[alias=L]| C\arrow[u, "f^*"]\arrow[ur, bend right=15pt, "\id"', ""{marking,name=Z,xshift=-2pt,yshift=4pt}]
						\arrow[from=u, to=Z, Rightarrow, "\epsilon"{description}]
				\end{tikzcd}
			\end{equation}
			By the triangle identity for the unit and counit of the adjunction $f_!\dashv f^*$, the pasting of the shaded subdiagram is the identity transformation $f^*\Rightarrow f^*$. Thus, the pasting $(\ref{diag:big-pasting})$ simplifies to the pasting of the unshaded portion on the right, which precisely gives the other composite in the diagram $(\ref{diag:cancel-unit-counit})$.
		\end{proof}
	\end{lemma}

	Finally, let us explain in which way Beck-Chevalley maps compose. All of the following statements are instances of \cite{kelly-street}*{Proposition~2.2}, and they are explicitly verified in \emph{loc.~cit.}~using similar pasting arguments to the one presented above.

	We begin with \emph{horizontal compositions} of Beck-Chevalley transformations:

	\begin{lemma}\label{lemma:BC-compose}
		Assume that all the vertical maps in the diagram
		\begin{equation*}
			\begin{tikzcd}
				A\arrow[r, "u^*"]\arrow[dr, Rightarrow, "\sigma"{description}] & B\arrow[r, "w^*"]\arrow[dr, Rightarrow, "\tau"{description}] & E\\
				C\arrow[u, "f^*"]\arrow[r, "v^*"'] & D\arrow[u, "g^*"{description}]\arrow[r, "x^*"'] & F\arrow[u, "h^*"']
			\end{tikzcd}
		\end{equation*}
		admit left adjoints. Then the composite
		\begin{equation*}
			\begin{tikzcd}
				h_!w^*u^*\arrow[r, "\BC_!(\tau)"] &[1em] x^*g_!u^*\arrow[r, "\BC_!(\sigma)"] &[1em] x^*v^*f_!
			\end{tikzcd}
		\end{equation*}
		of the Beck-Chevalley maps for the two little squares agrees with the Beck-Chevalley map of the total rectangle.\qed
	\end{lemma}

	Using this, we can establish one half of a `homotopy invariance' property of Beck-Chevalley transformations:

	\begin{corollary}\label{cor:BC-htpy-inv}
		Assume once more we are given a diagram $(\ref{diag:to-mate})$ whose vertical maps admit left adjoints, and assume we are further given equivalences of $1$-morphisms $f^{\prime*}\simeq f^*$ and $g^{\prime*}\simeq g^*$ (in particular, $f^{\prime*}$ and $g^{\prime*}$ again admit left adjoints). Then the Beck-Chevalley transformations fit into a commutative diagram
		\begin{equation*}
			\begin{tikzcd}
				g_!u^*\arrow[d, "\simeq"']\arrow[r, "\BC_!"] & v^*f_!\arrow[d, "\simeq"]\\
				g'_!u^*\arrow[r, "\BC_!"'] & v^*f'_!
			\end{tikzcd}
		\end{equation*}
		where the vertical equivalences are induced by the given equivalences between the right adjoints.
		\begin{proof}
			Applying the previous lemma iteratively to
			\begin{equation*}
				\begin{tikzcd}
					A\arrow[r, "\id"]\arrow[dr, Rightarrow,"\simeq"{description}] & A\arrow[r, "u^*"]\arrow[dr, Rightarrow, "\sigma"{description}] & B\arrow[r, "\id"]\arrow[dr,"\simeq"{description}, Rightarrow] & B\\
					C\arrow[u, "f^{\prime*}"]\arrow[r, "\id"'] & C\arrow[u, "f^*"]\arrow[r, "v^*"'] & D\arrow[u, "g^*"']\arrow[r,"\id"'] & D\arrow[u, "g^{\prime*}"']
				\end{tikzcd}
			\end{equation*}
			immediately yields the claim.
		\end{proof}
	\end{corollary}

	We can also compose Beck-Chevalley transformations vertically:

	\begin{lemma}
		Assume that all vertical maps in the diagram
		\begin{equation*}
			\begin{tikzcd}
				A\arrow[r, "u^*"]\arrow[dr, Rightarrow, "\sigma"{description}] & B\\
				C\arrow[u, "f^*"]\arrow[dr, Rightarrow, "\tau"{description}]\arrow[r, "v^*"{description}] & D\arrow[u, "g^*"']\\
				E\arrow[r, "w^*"']\arrow[u, "h^*"] & F\arrow[u, "k^*"']
			\end{tikzcd}
		\end{equation*}
		admit left adjoints. Then the composite
		\begin{equation*}
			\begin{tikzcd}
				k_!g_!u^*\arrow[r, "\BC_!(\sigma)"] &[1em] k_!v^*f_!\arrow[r, "\BC_!(\tau)"] &[1em] w^*h_!f_!
			\end{tikzcd}
		\end{equation*}
		agrees with the Beck-Chevalley map for the total rectangle, when we take the left adjoints of $f^*h^*$ and $g^*k^*$ to be $h_!f_!$ and $k_!g_!$, respectively, with the induced unit and counit.\qed
	\end{lemma}

	As in Corollary~\ref{cor:BC-htpy-inv} (or by direct inspection), we in particular conclude:

	\begin{corollary}\label{cor:BC-htpy-inv-vert}
		Assume once more we are given a diagram $(\ref{diag:to-mate})$ whose vertical maps admit left adjoints, and assume we are further given equivalences of $1$-morphisms $u^{\prime*}\simeq u^*$ and $v^{\prime*}\simeq v^*$. Then we have a commutative diagram
		\begin{equation*}
			\begin{tikzcd}[baseline=(X.base)]
				g_!u^*\arrow[r, "\BC_!"]\arrow[d, "\simeq"'] & v^*f_!\arrow[d, "\simeq"]\\
				g_!u^{\prime*}\arrow[r, "\BC_!"'] & |[alias=X]| v^{\prime*}f_!\rlap.
			\end{tikzcd}\qednow
		\end{equation*}
	\end{corollary}

	\bibliographystyle{amsalpha}
	\bibliography{reference}

\newcommand{\etalchar}[1]{$^{#1}$}
\providecommand{\bysame}{\leavevmode\hbox to3em{\hrulefill}\thinspace}
\providecommand{\MR}{\relax\ifhmode\unskip\space\fi MR }
\providecommand{\MRhref}[2]{%
  \href{http://www.ams.org/mathscinet-getitem?mr=#1}{#2}
}
\providecommand{\href}[2]{#2}
\begin{thebibliography}{BDG{\etalchar{+}}16}

\bibitem[Ada84]{Adams1984}
J.~F. Adams, \emph{Prerequisites (on equivariant stable homotopy) for
  {Carlsson}'s lecture}, Algebraic topology, {Proc}. {Conf}., {Aarhus} 1982,
  {Lect}. {Notes} {Math}. 1051, 483-532, 1984.

\bibitem[BDG{\etalchar{+}}16]{exposeI}
Clark Barwick, Emanuele Dotto, Saul Glasman, Denis Nardin, and Jay Shah,
  \emph{{Parametrized higher category theory and higher algebra: Exposé I --
  Elements of parametrized higher category theory}}, arXiv:1608.03657 (2016).

\bibitem[BDS16]{BDS2016Grothendieck}
Paul Balmer, Ivo Dell'Ambrogio, and Beren Sanders, \emph{Grothendieck-{Neeman}
  duality and the {Wirthm{\"u}ller} isomorphism}, Compos. Math. \textbf{152}
  (2016), no.~8, 1740--1776.

\bibitem[BEL21]{BEL2023Kasparov}
Ulrich Bunke, Alexander Engel, and Markus Land, \emph{A stable
  $\infty$-category for equivariant {$\KK$-theory}}, arXiv:2102.13372 (2021).

\bibitem[BH21]{BachmannHoyois2021Norms}
Tom Bachmann and Marc Hoyois, \emph{Norms in motivic homotopy theory},
  Ast{\'e}risque, vol. 425, Paris: Soci{\'e}t{\'e} Math{\'e}matique de France
  (SMF), 2021.

\bibitem[{Cis}19]{cisinski-book}
Denis-Charles {Cisinski}, \emph{{Higher Categories and Homotopical Algebra}},
  {Camb. Stud. Adv. Math.}, vol. 180, Cambridge: Cambridge University Press,
  2019.

\bibitem[CLL23a]{CLL_Global}
Bastiaan Cnossen, Tobias Lenz, and Sil Linskens, \emph{{Parametrized stability
  and the universal property of global spectra}}, arXiv:2301.08240 (2023).

\bibitem[CLL23b]{CLL_Clefts}
\bysame, \emph{Partial parametrized presentability and the universal property
  of equivariant spectra}, {arXiv:2307.11001} (2023).

\bibitem[Elm83]{elmendorf}
Anthony~D. Elmendorf, \emph{Systems of fixed point sets}, Trans. Amer. Math.
  Soc. \textbf{277} (1983), 275--284.

\bibitem[GH23]{Gepner_Heller}
David Gepner and Jeremiah Heller, \emph{The tom {Dieck} splitting theorem in
  equivariant motivic homotopy theory}, J. Inst. Math. Jussieu \textbf{22}
  (2023), no.~3, 1181--1250.

\bibitem[GM95]{greenlees-may}
John P.~C. Greenlees and J.~Peter May, \emph{Generalized {Tate} cohomology},
  Mem. Amer. Math. Soc., vol. 543, Providence, RI: American Mathematical
  Society (AMS), 1995.

\bibitem[Hau17]{hausmann-equivariant}
Markus Hausmann, \emph{{{\(G\)}}-symmetric spectra, semistability and the
  multiplicative norm}, J. Pure Appl. Algebra \textbf{221} (2017), no.~10,
  2582--2632.

\bibitem[Hau19]{hausmann-global}
\bysame, \emph{{Symmetric spectra model global homotopy theory of finite
  groups}}, {Algebr. Geom. Topol.} \textbf{19} (2019), no.~3, 1413--1452.

\bibitem[HL13]{hopkins2013ambidexterity}
Michael Hopkins and Jacob Lurie, \emph{Ambidexterity in {$K(n)$}-local stable
  homotopy theory},
  \url{https://people.math.harvard.edu/~lurie/papers/Ambidexterity.pdf} (2013).

\bibitem[Hoy17]{Hoyois2017Motivic}
Marc Hoyois, \emph{The six operations in equivariant motivic homotopy theory},
  Adv. Math. \textbf{305} (2017), 197--279.

\bibitem[Hu03]{Hu2003Duality}
Po~Hu, \emph{Duality for smooth families in equivariant stable homotopy
  theory}, Ast{\'e}risque, vol. 285, Paris: Soci{\'e}t{\'e} Math{\'e}matique de
  France (SMF), 2003.

\bibitem[Kal11]{DerMack11}
D.~Kaledin, \emph{Derived {Mackey} functors}, Mosc. Math. J. \textbf{11}
  (2011), no.~4, 723--803.

\bibitem[KS74]{kelly-street}
G.~Maxwell Kelly and Ross Street, \emph{Review of the elements of
  2-categories}, Category {Sem}., {Proc}., {Sydney} 1972/1973, {Lect}. {Notes}
  {Math}. 420, 1974, pp.~75--103.

\bibitem[Len20]{g-global}
Tobias Lenz, \emph{{$G$-Global Homotopy Theory and Algebraic $K$-Theory}}, {to
  appear in \textit{Mem.~Amer.~Math.~Soc.}, arXiv:2012.12676} (2020).

\bibitem[LMS86]{LMS}
L.~Gaunce {Lewis, Jr.}, J.~Peter {May}, and Mark {Steinberger},
  \emph{{Equivariant Stable Homotopy Theory. With contributions by J. E.
  McClure}}, {Lect. Notes Math.}, vol. 1213, Springer, Berlin Heidelberg New
  York London Paris Tokyo, 1986.

\bibitem[L{\"u}c05]{lueck}
Wolfgang L{\"u}ck, \emph{Survey on classifying spaces for families of
  subgroups}, Infinite groups: geometric, combinatorial and dynamical aspects.
  Based on the international conference on group theory: geometric,
  combinatorial and dynamical aspects of infinite groups, Gaeta, Italy, June
  1--6, 2003, Basel: Birkh{\"a}user, 2005, pp.~269--322.

\bibitem[Lur09]{HTT}
Jacob Lurie, \emph{Higher topos theory}, Ann.~Math.~Stud., vol. 170, Princeton
  University Press, Princeton, NJ, 2009, updated version available at
  \url{https://people.math.harvard.edu/~lurie/papers/highertopoi.pdf}.

\bibitem[Lur24]{kerodon}
\bysame, \emph{Kerodon}, {online textbook, available at
  \url{https://kerodon.net}} (2024).

\bibitem[Mar21]{martini2021yoneda}
Louis Martini, \emph{Yoneda's lemma for internal higher categories},
  arXiv:2103.17141 (2021).

\bibitem[May03]{wirthrev}
J.~Peter May, \emph{The {Wirthm{\"u}ller} isomorphism revisited}, Theory Appl.
  Categ. \textbf{11} (2003), 132--142.

\bibitem[MM02]{mandell-may}
Michael~A. Mandell and J.~Peter May, \emph{Equivariant orthogonal spectra and
  {{\(S\)}}-modules}, Mem. Amer. Math. Soc., vol. 755, Providence, RI: American
  Mathematical Society (AMS), 2002.

\bibitem[MNN17]{MNN2017Nilpotence}
Akhil Mathew, Niko Naumann, and Justin Noel, \emph{Nilpotence and descent in
  equivariant stable homotopy theory}, Adv. Math. \textbf{305} (2017),
  994--1084.

\bibitem[MW22]{martiniwolf2022presentable}
Louis Martini and Sebastian Wolf, \emph{Presentable categories internal to an
  $\infty$-topos}, arXiv:2209.05103 (2022).

\bibitem[MW24]{martiniwolf2021limits}
\bysame, \emph{Colimits and cocompletions in internal higher category theory},
  High. Struct. \textbf{8} (2024), no.~1, 97--192.

\bibitem[Nar16]{nardin2016exposeIV}
Denis Nardin, \emph{{Parametrized higher category theory and higher algebra:
  Expos\'e IV -- Stability with respect to an orbital $\infty$-category}},
  arXiv:1608.07704 (2016).

\bibitem[Ost16]{ostermayr}
Dominik Ostermayr, \emph{Equivariant {{\(\Gamma\)}}-spaces}, Homology Homotopy
  Appl. \textbf{18} (2016), no.~1, 295--324.

\bibitem[Pow90]{power-pasting}
A.~John Power, \emph{A 2-categorical pasting theorem}, J. Algebra \textbf{129}
  (1990), no.~2, 439--445.

\bibitem[QS21]{QuigleyShay2021Tate}
J.D. Quigley and Jay Shah, \emph{On the parametrized {T}ate construction},
  arXiv:2110.07707 (2021).

\bibitem[RV16]{ReichVarisco2016Adams}
Holger Reich and Marco Varisco, \emph{On the {Adams} isomorphism for
  equivariant orthogonal spectra}, Algebr. Geom. Topol. \textbf{16} (2016),
  no.~3, 1493--1566.

\bibitem[San19]{Sanders2019Adams}
Beren Sanders, \emph{The compactness locus of a geometric functor and the
  formal construction of the {Adams} isomorphism}, J. Topol. \textbf{12}
  (2019), no.~2, 287--327.

\bibitem[Sch18]{schwede2018global}
Stefan Schwede, \emph{Global homotopy theory}, {New Math. Monogr.}, vol.~34,
  Cambridge University Press, 2018.

\bibitem[Sha23]{shah2021parametrized}
Jay Shah, \emph{Parametrized higher category theory}, Algebr. Geom. Topol.
  \textbf{23} (2023), no.~2, 509--644.

\bibitem[{Shi}89]{shimakawa}
Kazuhisa {Shimakawa}, \emph{{Infinite loop $G$-spaces associated to monoidal
  $G$-graded categories}}, {Publ. Res. Inst. Math. Sci.} \textbf{25} (1989),
  no.~2, 239--262.

\bibitem[{Shi}91]{shimakawa-simplify}
\bysame, \emph{{A note on $\Gamma_G$-spaces}}, Osaka J. Math. \textbf{28}
  (1991), no.~2, 223--228.

\bibitem[{Ste}16]{cellular}
Marc {Stephan}, \emph{{On equivariant homotopy theory for model categories}},
  {Homology Homotopy Appl.} \textbf{18} (2016), no.~2, 183--208.

\bibitem[{Tig}22]{tigilauri}
Giorgi {Tigilauri}, \emph{{Global Adams isomorphism}}, Master's thesis,
  {Rheinische Friedrich-Wilhelms-Universität Bonn}, Bonn, September 2022.

\end{thebibliography}
\end{document}